\documentclass[a4paper]{article}
\usepackage[dvips]{graphicx}
\usepackage{amssymb}
\usepackage{amsmath}
\usepackage{mathrsfs}
\usepackage{color}
\usepackage{theorem}
\usepackage{comment}
\usepackage{url}
\usepackage{ascmac}
\newcommand{\pref}[1]{(\ref{#1})}
\newcommand{\be}{\begin{equation}}
\newcommand{\ee}{\end{equation}}
\newcommand{\qed}{{\unskip\nobreak\hfil\penalty50\quad\null\nobreak\hfil
	$\square$\parfillskip0pt\finalhyphendemerits0\par\medskip}}

\newcommand{\vecf}{\mbox{\boldmath $ f $}}

\newcommand{\vectau}{\mbox{\boldmath $ \tau $}}

\renewcommand{\epsilon}{\varepsilon}

\newtheorem{thm}{Theorem}[section]
\newtheorem{prop}{Proposition}[section]
\newtheorem{lem}{Lemma}[section]
\newtheorem{cor}{Corollary}[section]

\theorembodyfont{\rmfamily}
\newtheorem{rem}{Remark}[section]
\newtheorem{proof}{\normalfont\itshape Proof.}

\title{Upper and lower bounds and modulus of continuity of decomposed M\"{o}bius energies}
\author{
Aya Ishizeki\thanks{Chiba University,
Japan,
e-mail:
a.ishizeki@chiba-u.jp,
supported by KAKENHI (17J01429)}
\ \&
Takeyuki Nagasawa\thanks{Saitama University,
Japan,
e-mail:
tnagasaw@rimath.saitama-u.ac.jp,
supported by KAKENHI (17K05310).}}
\date{}
\begin{document}
\maketitle
\begin{abstract}
The M\"{o}bius energy is one of the knot energies, and is named after its M\"{o}bius invariant property.
It is known to have several different expressions.
One is in terms of the cosine of conformal angle,
and is called the cosine formula.
Another is the decomposition into M\"{o}bius invariant parts, called the decomposed M\"{o}bius energies.
Hence the cosine formula is the sum of the decomposed energies.
This raises a question.
Can each of the decomposed energies be estimated by the cosine formula~?
Here we give an affirmative answer:
the upper and lower bounds,
and modulus of continuity of decomposed parts can be evaluated in terms of the cosine formula.
In addition,
we provide estimates of the difference in decomposed energies between the two curves in terms of M\"{o}bius invariant quantities.
\end{abstract}
\section{Introduction}
\par
Let $ \vecf \, : \, \mathbb{R} / \mathcal{L} \mathbb{Z} \to \mathbb{R}^n $ be an arch-length parametrization of a closed curve with the total length $ \mathcal{L} $ embedded in $ \mathbb{R}^n $.
There are two distances between $ \vecf ( s_1 ) $ and $ \vecf ( s_2 ) $;
one is the extrinsic distance $ \| \Delta \vecf \| = \| \vecf ( s_1 ) - \vecf ( s_2 ) \|_{ \mathbb{R}^n } $,
and the other is the intrinsic distance $ | \Delta s | = \mathrm{dist}_{ \mathbb{R} / \mathcal{L} \mathbb{Z} } ( s_1 , s_2 ) $,
{\it i.e.},
the shortest distance along the curve.
The M\"{o}bius energy $ \mathcal{E} $ of $ \vecf $ is defined as
\[
	\mathcal{E} ( \vecf )
	=
	\iint_{ ( \mathbb{R} / \mathcal{L} \mathbb{Z} )^2 }
	\mathscr{M} ( \vecf ) \, d s_1 d s_2
	,
\]
where
\[
	\mathscr{M} ( \vecf )
	=
	\frac 1 { \| \Delta \vecf \|^2 }
	-
	\frac 1 { | \Delta s |^2 } .
\]
The energy $ \mathcal{E} $ is one of O'Hara's energies (\cite{OH1}),
and is named after the invariance under M\"{o}bius transformation,
which was proved by Freedman-He-Wang \cite{FHW}.
It has other expressions.
We can find
\be
	\mathcal{E} ( \vecf )
	=
	\mathcal{E}_0 ( \vecf ) + 4
	\label{cosine formula}
\ee
with
\[
	\mathcal{E}_0 ( \vecf )
	=
	\iint_{ ( \mathbb{R} / \mathcal{L} \mathbb{Z} )^2 }
	\mathscr{M}_0 ( \vecf ) \, d s_1 d s_2
	,
	\quad
	\mathscr{M}_0 ( \vecf )
	=
	\frac { 1 - \cos \varphi } { \| \Delta \vecf \|^2 }
\]
in \cite{KS}.
Here $ \varphi $ is the {\it conformal angle} defined as follows.
Let $ C_{12} $ be the circle contacting a knot $ \mathrm{Im} \vecf $ at $ \vecf ( s_1 ) $ and passing through $ \vecf ( s_2 ) $.
We define the circle $ C_{21} $ similarly.
The angle $ \varphi ( s_1 , s_2 ) $ is that between these two circles at $ \vecf ( s_1 ) $ (and also at $ \vecf ( s_2 ) $).
Since it is M\"{o}bius invariant,
the M\"{o}bius invariant property of $ \mathcal{E} $ can be easily read from the {\it cosine formula} \pref{cosine formula}.
Another expression of $ \mathcal{E} $ was shown by the authors in \cite{IshizekiNagasawaI}:
\be
	\mathcal{E} ( \vecf )
	=
	\mathcal{E}_1 ( \vecf ) + \mathcal{E}_2 ( \vecf ) + 4
	\label{decomposition}
\ee
with
\begin{gather*}
	\mathcal{E}_i ( \vecf )
	=
	\iint_{ ( \mathbb{R} / \mathcal{L} \mathbb{Z} )^2 }
	\mathscr{M}_i ( \vecf ) \, d s_1 d s_2
	\quad ( i = 1 , \, 2 )
	,
	\\
	\mathscr{M}_1 ( \vecf )
	=
	\frac { \| \Delta \vectau \|^2 } { 2 \| \Delta \vecf \|^2 }
	,
	\quad
	\mathscr{M}_2 ( \vecf )
	=
	\frac 2 { \| \Delta \vecf \|^2 }
	\left\langle
	\vectau ( s_1 ) \wedge \frac { \Delta \vecf } { \| \Delta \vecf \| }
	,
	\vectau ( s_2 ) \wedge \frac { \Delta \vecf } { \| \Delta \vecf \| }
	\right\rangle
	.
\end{gather*}
Both $ \mathcal{E}_1 $ and $ \mathcal{E}_2 $ are M\"{o}bius invariant energies.
It holds not only that
\be
	\mathcal{E}_0 ( \vecf )
	=
	\mathcal{E}_1 ( \vecf )
	+
	\mathcal{E}_2 ( \vecf )
	\label{E0=E1+E2}
\ee
but also that
\be
	\mathscr{M}_0 ( \vecf )
	=
	\mathscr{M}_1 ( \vecf )
	+
	\mathscr{M}_2 ( \vecf )
	.
	\label{M0=M1+M2}
\ee
See Lemma \ref{M_f} for \pref{M0=M1+M2}.
Hence
$ \mathcal{E}_0 ( \vecf ) $ can be evaluated from the decomposed energies $ \mathcal{E}_1 ( \vecf ) $ and $ \mathcal{E}_2 ( \vecf ) $.
Since $ \mathcal{E}_2 ( \vecf ) $ is not necessarily non-negative,
the converse estimate is not so obvious.
In this paper,
we consider this and related problems.
\par
It is known that if $ \mathcal{E}_0 ( \vecf ) < \infty $,
then $ \vecf $ is bi-Lipschitz,
that is,
$ \displaystyle{ 	\sup_{ \substack{ ( s_1 , s_2 ) \in ( \mathbb{R} / \mathcal{L} \mathbb{Z} )^2 \\ s_1 \ne s_2 } } \frac { | \Delta s | } { \| \Delta \vecf \| } } $ is bounded;
see \cite{Blatt201202}.
This quantity is called the {\it distortion};
see \cite{Gromov,OH2}.
In this paper we use
\[
	X( s_1 , s_2 ) = \log \frac { | \Delta s |^2 } { \| \Delta \vecf \|^2 }
\]
instead of the distortion.
Since $ \| \Delta \vecf \| \leqq | \Delta s | $,
the function $ X $ is non-negative.
We will give upper and lower bounds (Theorem \ref{upper and lower bounds}),
and the modulus of continuity (Theorem \ref{modulus_of_continuity}) of the decomposed energies by use of $ \mathcal{E}_0 ( \vecf ) $ and $ X $ in \S\S~2 and 3 respectively.
\par
Let $ \vecf $ and $ \widetilde { \vecf } $ be the parametrizations of two closed curves embedded in $ \mathbb{R}^n $.
Taking the M\"{o}bius invariance of $ \mathcal{E}_1 $ and $ \mathcal{E}_2 $ into consideration,
we should use difference between certain M\"{o}bius invariances when we estimate the energy difference $ \mathcal{E}_i ( \vecf ) - \mathcal{E}_i ( \widetilde { \vecf } ) $.
Here we shall use
\[
	\mathscr{C} ( \vecf ) =
	\frac { \| \dot { \vecf } ( \theta_1 ) \| \| \dot { \vecf } ( \theta_2 ) \| }
	{ \| \Delta \vecf \|^2 }
\]
in Theorem \ref{Diff C},
which is shown in \S~4.
The M\"{o}bius invariance of $ \mathscr{C} $ follows from that of the cross ratio.
Note that the conformal angle $ \varphi $ can be written by $ \mathscr{C} $ and its second derivative. 
\section{Upper and lower bounds}
\par
Firstly,
we observe that $ \mathscr{M} ( \vecf ) $,
$ \mathscr{M}_1 ( \vecf ) $ and
$ \mathscr{M}_2 ( \vecf ) $ can be written by use of 
$ \mathscr{M}_0 ( \vecf ) $ and derivatives of $ X $ for $ \displaystyle{ 0 < | s_1 - s_2 | < \frac { \mathcal{L} } 2 } $.
\begin{rem}
It follows from $ \mathcal{E} ( \vecf ) < \infty $ that $ \vecf \in W^{ \frac 12 , 2 } ( \mathbb{R} / \mathcal{L} \mathbb{Z} ) $;
see \cite{Blatt201202}.
Consequently
$ \vecf^\prime = \vectau $ exists a.e.\ $ s \in \mathbb{R} / \mathcal{L} \mathbb{Z} $.
Hence $ \displaystyle{ \frac { \partial^2 } { \partial s_1 \partial s_2 } \| \vecf ( s_1 ) - \vecf ( s_2 ) \|^2 } $ can be defined for a.e.\ $ ( s_1 , s_2 ) \in ( \mathbb{R} / \mathcal{L} \mathbb{Z} )^2 $.
However $ \displaystyle{ \frac { \partial^2 } { \partial s_1 \partial s_2 } | \Delta s |^2 } $ cannot be defined at $ \displaystyle{ | s_1 - s_2 | = \frac { \mathcal{L} } 2 } $ as a Sobolev function.
Furthermore,
$ X $ is not defined when $ s_1 = s_2 $.
Thus,
we consider derivatives for $ \displaystyle{ 0 < | s_1 - s_2 | < \frac { \mathcal{L} } 2 } $.
\end{rem}
\begin{lem}
For $ \displaystyle{ 0 < | s_1 - s_2 | < \frac { \mathcal{L} } 2 } $,
we can set $ \Delta s = s_1 - s_2 $.
Then,
it holds that
\begin{align}
	2 \mathscr{M} ( \vecf )
	= & \
	2 \mathscr{M}_0 ( \vecf )
	-
	\frac { \partial^2 X } { \partial s_1 \partial s_2 }
	\label{M}
	,
	\\
	2 \mathscr{M}_1 ( \vecf )
	= & \
	2 \mathscr{M}_0 ( \vecf )
	-
	2
	\frac { \partial^2 X } { \partial s_1 \partial s_2 }
	+
	\frac { \partial X } { \partial s_1 }
	\frac { \partial X } { \partial s_2 }
	+
	\frac 2 { \Delta s }
	\left(
	\frac { \partial X } { \partial s_1 }
	-
	\frac { \partial X } { \partial s_2 }
	\right)
	\label{M1}
	,
	\\
	2 \mathscr{M}_2 ( \vecf )
	= & \
	2
	\frac { \partial^2 X } { \partial s_1 \partial s_2 }
	-
	\frac { \partial X } { \partial s_1 }
	\frac { \partial X } { \partial s_2 }
	-
	\frac 2 { \Delta s }
	\left(
	\frac { \partial X } { \partial s_1 }
	-
	\frac { \partial X } { \partial s_2 }
	\right)
	.
	\label{M2}
\end{align}
\label{M_f}
\end{lem}
\begin{proof}
By the elemental calculation,
we can see that the cosine of the conformal angle is
\[
	\cos \varphi ( s_1 , s_2 )
	=
	\frac 12
	\| \Delta \vecf \|^2
	\frac { \partial^2 } { \partial s_1 \partial s_2 }
	\log \| \Delta \vecf \|^2
	.
\]
In addition,
\be
	\frac \partial { \partial s_1 } \log | \Delta s |^2
	=
	\frac 2 { \Delta s }
	,
	\quad
	\frac \partial { \partial s_2 } \log | \Delta s |^2
	=
	- \frac 2 { \Delta s }
	,
	\quad
	\frac { \partial^2 } { \partial s_1 \partial s_2 } \log | \Delta s |^2
	=
	\frac 2 { | \Delta s |^2 }
	\label{D log Ds^2}
\ee
are also elementary.
Hence we have
\begin{align*}
	\frac { \partial^2 X } { \partial s_1 \partial s_2 }
	= & \
	\frac { \partial^2 } { \partial s_1 \partial s_2 }
	\left( \log | \Delta s |^2 - \log \| \Delta \vecf \|^2 \right)
	\\
	= & \
	-
	2
	\left(
	\frac 1 { \| \Delta \vecf \|^2 }
	-
	\frac 1 { | \Delta s |^2 }
	-
	\frac { 1 - \cos \varphi } { \| \Delta \vecf \|^2 }
	\right)
	\\
	= & \
	- 2 \left( \mathscr{M} ( \vecf ) - \mathscr{M}_0 ( \vecf ) \right)
	.
\end{align*}
It holds that
\begin{align}
	&
	\frac 1 { \| \Delta \vecf \|^2 }
	\frac { \partial^2 } { \partial s_1 \partial s_2} \| \Delta \vecf \|^2
	\nonumber
	\\
	& \quad
	=
	\frac { \partial^2 } { \partial s_1 \partial s_2 } \log \| \Delta \vecf \|^2
	+
	\left( \frac \partial { \partial s_1 } \log \| \Delta \vecf \|^2 \right)
	\left( \frac \partial { \partial s_2 } \log \| \Delta \vecf \|^2 \right)
	\label{D^2 Df^2}
	.
\end{align}
Using these,
we obtain
\begin{align*}
	2 \mathscr{M}_1 ( \vecf )
	= & \
	\frac { \| \Delta \vectau \|^2 } { \| \Delta \vecf \|^2 }
	=
	\frac { 2 - 2 \vectau ( s_1 ) \cdot \vectau ( s_2 ) } { \| \Delta \vecf \|^2 }
	\\
	= & \
	\frac 1 { \| \Delta \vecf \|^2 }
	\left( 2 + \frac { \partial^2 } { \partial s_1 \partial s_2 } \| \Delta \vecf \|^2 \right)
	\\
	= & \
	\frac 2 { \| \Delta \vecf \|^2 }
	+
	\frac { \partial^2 } { \partial s_1 \partial s_2 } \log \| \Delta \vecf \|^2
	+
	\left( \frac \partial { \partial s_1 } \log \| \Delta \vecf \|^2 \right)
	\left( \frac \partial { \partial s_2 } \log \| \Delta \vecf \|^2 \right)
	\\
	= & \
	\frac 2 { \| \Delta \vecf \|^2 }
	+
	\frac { \partial^2 } { \partial s_1 \partial s_2 } \log \frac { \| \Delta \vecf \|^2 } { | \Delta s |^2 }
	+
	\frac { \partial^2 } { \partial s_1 \partial s_2 } \log | \Delta s |^2
	\\
	& \quad
	+ \,
	\left(
	\frac \partial { \partial s_1 } \log \frac { \| \Delta \vecf \|^2 } { | \Delta s |^2 }
	+
	\frac \partial { \partial s_1 } \log | \Delta s |^2
	\right)
	\left(
	\frac \partial { \partial s_2 } \log \frac { \| \Delta \vecf \|^2 } { | \Delta s |^2 }
	+
	\frac \partial { \partial s_2 } \log | \Delta s |^2
	\right)
	\\
	= & \
	\frac 2 { \| \Delta \vecf \|^2 }
	-
	\frac { \partial^2 X } { \partial s_1 \partial s_2 }
	+
	\frac 2 { | \Delta s |^2 }
	+
	\left(
	- \frac { \partial X } { \partial s_1 } + \frac 2 { \Delta s }
	\right)
	\left(
	- \frac { \partial X } { \partial s_2 } - \frac 2 { \Delta s }
	\right)
	\\
	= & \
	2 \left( \frac 1 { \| \Delta \vecf \|^2 } - \frac 1 { | \Delta s |^2 } \right)
	-
	\frac { \partial^2 X } { \partial s_1 \partial s_2 }
	+
	\frac { \partial X } { \partial s_1 } \frac { \partial X } { \partial s_2 }
	+
	\frac 2 { \Delta s }
	\left( \frac { \partial X } { \partial s_1 } - \frac { \partial X } { \partial s_2 } \right)
	\\
	= & \
	2 \mathscr{M} ( \vecf )
	+
	\frac { \partial^2 X } { \partial s_1 \partial s_2 }
	-
	2 \frac { \partial^2 X } { \partial s_1 \partial s_2 }
	+
	\frac { \partial X } { \partial s_1 } \frac { \partial X } { \partial s_2 }
	+
	\frac 2 { \Delta s }
	\left( \frac { \partial X } { \partial s_1 } - \frac { \partial X } { \partial s_2 } \right)
	\\
	= & \
	2 \mathscr{M}_0 ( \vecf )
	-
	2 \frac { \partial^2 X } { \partial s_1 \partial s_2 }
	+
	\frac { \partial X } { \partial s_1 } \frac { \partial X } { \partial s_2 }
	+
	\frac 2 { \Delta s }
	\left( \frac { \partial X } { \partial s_1 } - \frac { \partial X } { \partial s_2 } \right)
	.
\end{align*}
By the definition of the inner product of 2-vectors,
we have	
\begin{align*}
	&
	\left\langle
	\vectau ( s_1 ) \wedge \frac { \Delta \vecf } { \| \Delta \vecf \| }
	,
	\vectau ( s_2 ) \wedge \frac { \Delta \vecf } { \| \Delta \vecf \| }
	\right\rangle
	\\
	& \quad
	=
	\vectau ( s_1 ) \cdot \vectau ( s_2 )
	-
	\left( \vectau ( s_1 ) \cdot \frac { \Delta \vecf } { \| \Delta \vecf \| } \right)
	\left( \vectau ( s_2 ) \cdot \frac { \Delta \vecf } { \| \Delta \vecf \| } \right)
	\\
	& \quad
	=
	-
	\left( \frac \partial { \partial s_1 } \Delta \vecf \right)
	\cdot
	\left( \frac \partial { \partial s_2 } \Delta \vecf \right)
	+
	\left( \frac \partial { \partial s_1 } \| \Delta \vecf \| \right)
	\left( \frac \partial { \partial s_2 } \| \Delta \vecf \| \right)
	\\
	& \quad
	=
	-
	\frac 12 \frac { \partial^2 } { \partial s_1 \partial s_2 } \| \Delta \vecf \|^2
	+
	\left( \frac \partial { \partial s_1 } \| \Delta \vecf \| \right)
	\left( \frac \partial { \partial s_2 } \| \Delta \vecf \| \right)
	.
\end{align*}
Combining \pref{D^2 Df^2},
\pref{D log Ds^2} with
\begin{align*}
	\frac 1 { \| \Delta \vecf \|^2 }
	\left( \frac \partial { \partial s_1 } \| \Delta \vecf \| \right)
	\left( \frac \partial { \partial s_2 } \| \Delta \vecf \| \right)
	= & \
	\frac 14
	\left( \frac \partial { \partial s_1 } \log \| \Delta \vecf \|^2 \right)
	\left( \frac \partial { \partial s_2 } \log \| \Delta \vecf \|^2 \right)
	,
\end{align*}
we have
\begin{align*}
	2 \mathscr{M}_2 ( \vecf )
	= & \
	\frac 4 { \| \Delta \vecf \|^2 }
	\left\langle
	\vectau ( s_1 ) \wedge \frac { \Delta \vecf } { \| \Delta \vecf \| }
	,
	\vectau ( s_2 ) \wedge \frac { \Delta \vecf } { \| \Delta \vecf \| }
	\right\rangle
	\\
	= & \
	-
	2 \frac { \partial^2 } { \partial s_1 \partial s_2 } \log \| \Delta \vecf \|^2
	-
	\left( \frac \partial { \partial s_1 } \log \| \Delta \vecf \|^2 \right)
	\left( \frac \partial { \partial s_2 } \log \| \Delta \vecf \|^2 \right)
	\\
	= & \
	- 2
	\left(
	\frac { \partial^2 } { \partial s_1 \partial s_2 }
	\log \frac { \| \Delta \vecf \|^2 } { | \Delta s |^2 }
	+
	\frac 2 { | \Delta s |^2 }
	\right)
	\\
	& \quad
	- \,
	\left( \frac \partial { \partial s_1 } \log \frac { \| \Delta \vecf \|^2 } { | \Delta s |^2 } + \frac 2 { \Delta s } \right)
	\left( \frac \partial { \partial s_2 } \log \frac { \| \Delta \vecf \|^2 } { | \Delta s |^2 } - \frac 2 { \Delta s } \right)
	\\
	= & \
	2
	\frac { \partial^2 X } { \partial s_1 \partial s_2 }
	-
	\frac { \partial X } { \partial s_1 }
	\frac { \partial X } { \partial s_2 }
	-
	\frac 2 { \Delta s }
	\left(
	\frac { \partial X } { \partial s_1 }
	-
	\frac { \partial X } { \partial s_2 }
	\right)
	.
\end{align*}
\qed
\end{proof}
\begin{prop}
If $ \mathcal{E}_0 ( \vecf ) < \infty $,
then $ X \mathscr{M} ( \vecf ) $,
$ X \mathscr{M}_0 ( \vecf ) $,
$ \displaystyle{ \frac X { | \Delta s |^2 } } \in L^1 ( ( \mathbb{R} / \mathcal{L} \mathbb{Z} )^2 ) $,
$ X \left( \cdot + \frac { \mathcal{L} } 2 , \cdot \right) \in L^1 ( \mathbb{R} / \mathcal{L} \mathbb{Z} ) $,
and it holds that
\begin{align}
	\mathcal{E}_1 ( \vecf )
	= & \
	\mathcal{E}_0 ( \vecf )
	+
	\iint_{ ( \mathbb{R} / \mathcal{L} \mathbb{Z} )^2 }
	X
	\left( \mathscr{M} ( \vecf ) - \mathscr{M}_0 ( \vecf )
	+
	\frac 2 { | \Delta s |^2 }
	\right)
	d s_1 d s_2
	\nonumber
	\\
	& \quad
	- \,
	\frac 4 { \mathcal{L} }
	\int_{ \mathbb{R} / \mathcal{L} \mathbb{Z} }
	X \left( s + \frac { \mathcal{L} } 2 , s \right) ds
	+ 8
	,
	\label{E1}
	\\
	\mathcal{E}_2 ( \vecf )
	= & \
	-
	\iint_{ ( \mathbb{R} / \mathcal{L} \mathbb{Z} )^2 }
	X
	\left( \mathscr{M} ( \vecf ) - \mathscr{M}_0 ( \vecf )
	+
	\frac 2 { | \Delta s |^2 }
	\right)
	d s_1 d s_2
	\nonumber
	\\
	& \quad
	+ \,
	\frac 4 { \mathcal{L} }
	\int_{ \mathbb{R} / \mathcal{L} \mathbb{Z} }
	X \left( s + \frac { \mathcal{L} } 2 , s \right) ds
	- 8
	.
	\label{E2}
\end{align}
\end{prop}
\begin{proof}
In \cite{IshizekiNagasawaI},
the authors showed that $ \mathcal{E}_0 ( \vecf ) < \infty $ implies the absolute integrability of $ \mathscr{M}_1 ( \vecf ) $ and $ \mathscr{M}_2 ( \vecf ) $. 
Hence
\[
	\mathcal{E}_i ( \vecf )
	=
	\lim_{ \substack{ \epsilon \to + 0 \\ \delta \to + 0 } }
	\iint_{ \epsilon \leqq | s_1 - s_2 | \leqq \frac { \mathcal{L} } 2 - \delta }
	\mathscr{M}_i ( \vecf ) \, d s_1 d s_2
	.
\]
Set
\begin{align*}
	I_1 ( \epsilon , \delta )
	= & \
	\iint_{ \epsilon \leqq | s_1 - s_2 | \leqq \frac { \mathcal{L} } 2 - \delta }
	\frac { \partial^2 X } { \partial s_1 \partial s_2 }
	\, d s_1 d s_2 ,
	\\
	I_2 ( \epsilon , \delta )
	= & \
	\iint_{ \epsilon \leqq | s_1 - s_2 | \leqq \frac { \mathcal{L} } 2 - \delta }
	\frac { \partial X } { \partial s_1 }
	\frac { \partial X } { \partial s_2 }
	\, d s_1 d s_2 ,
	\\
	I_3 ( \epsilon , \delta )
	= & \
	\iint_{ \epsilon \leqq | s_1 - s_2 | \leqq \frac { \mathcal{L} } 2 - \delta }
	\frac 2 { \Delta s }
	\left(
	\frac { \partial X } { \partial s_1 }
	-
	\frac { \partial X } { \partial s_2 }
	\right)
	\, d s_1 d s_2
	.
\end{align*}
By Lemma \ref{M_f},
it is enough for the proof to show
\begin{align}
	&
	\lim_{ \substack{ \epsilon \to + 0 \\ \delta \to + 0 } }
	I_1 ( \epsilon , \delta )
	=
	- 8
	\label{lim I_1}
	,
	\\
	&
	\lim_{ \substack{ \epsilon \to + 0 \\ \delta \to + 0 } }
	\left(
	I_2 ( \epsilon , \delta ) + I_3 ( \epsilon , \delta )
	\right)
	\nonumber
	\\
	& \quad
	=
	2
	\iint_{ ( \mathbb{R} / \mathcal{L} \mathbb{Z} )^2 }
	X \left(
	\mathscr{M} ( \vecf ) - \mathscr{M}_0 ( \vecf ) + \frac 2 { | \Delta s |^2 }
	\right)
	d s_1 d s_2
	-
	\frac 8 { \mathcal{L} }
	\int_{ \mathbb{R} / \mathcal{L} \mathbb{Z} }
	X \left( s + \frac { \mathcal{L} } 2 , s \right) ds
	\label{lim I_2}
	.
\end{align}
\par
We assume the boundedness of $ \mathcal{E} ( \vecf ) $ and $ \mathcal{E}_0 ( \vecf ) $.
Since $ \mathscr{M} ( \vecf ) $ and $ \mathscr{M}_0 ( \vecf ) $ are non-negative,
these are absolutely integrable.
Consequently,
\pref{M} and \pref{E0=E1+E2} implies
\begin{align*}
	\lim_{ \substack{ \epsilon \to + 0 \\ \delta \to + 0 } }
	I_1 ( \epsilon , \delta )
	= & \
	-
	2
	\iint_{ ( \mathbb{R} / \mathcal{L} \mathbb{Z} )^2 }
	\left( \mathscr{M} ( \vecf ) - \mathscr{M}_0 ( \vecf ) \right)
	d s_1 d s_2
	\\
	= & \
	-
	2
	\left(
	\mathcal{E} ( \vecf )
	-
	\mathcal{E}_0 ( \vecf )
	\right)
	=
	- 8
	.
\end{align*}
\par
We have
\begin{align*}
	I_2 ( \epsilon , \delta )
	= & \
	\int_{ \mathbb{R} / \mathcal{L} \mathbb{Z} }
	\left(
	\int_{ s_2 - \frac { \mathcal{L} } 2 + \delta }^{ s_2 - \epsilon }
	+
	\int_{ s_2 + \epsilon }^{ s_2 + \frac { \mathcal{L} } 2 - \delta }
	\right)
	\frac { \partial X } { \partial s_1 }
	\frac { \partial X } { \partial s_2 }
	\,
	d s_1 d s_2
	\\
	= & \
	\int_{ \mathbb{R} / \mathcal{L} \mathbb{Z} }
	\left(
	\left[
	X
	\frac { \partial X } { \partial s_2 }
	\right]_{ s_1 = s_2 - \frac { \mathcal{L} } 2 + \delta }^{ s_1 = s_2 - \epsilon }
	+
	\left[
	X
	\frac { \partial X } { \partial s_2 }
	\right]_{ s_1 = s_2 + \epsilon }^{ s_1 = s_2 + \frac { \mathcal{L} } 2 - \delta }
	\right)
	d s_2
	\\
	& \quad
	- \,
	\int_{ \mathbb{R} / \mathcal{L} \mathbb{Z} }
	\left(
	\int_{ s_2 - \frac { \mathcal{L} } 2 + \delta }^{ s_2 - \epsilon }
	+
	\int_{ s_2 + \epsilon }^{ s_2 + \frac { \mathcal{L} } 2 - \delta }
	\right)
	X
	\frac { \partial^2 X } { \partial s_1 \partial s_2 }
	\,
	d s_1 d s_2
	\\
	= & \
	-
	\int_{ \mathbb{R} / \mathcal{L} \mathbb{Z} }
	\left(
	\left[
	X
	\frac { \partial X } { \partial s_2 }
	\right]_{ s_1 = s_2 - \epsilon }^{ s_1 = s_2 + \epsilon }
	-
	\left[
	X
	\frac { \partial X } { \partial s_2 }
	\right]_{ s_1 = s_2 - \frac { \mathcal{L} } 2 + \delta }^{ s_1 = s_2 + \frac { \mathcal{L} } 2 - \delta }
	\right)
	d s_2
	\\
	& \quad
	- \,
	\iint_{ \epsilon \leqq | s_1 - s_2 | \leqq \frac { \mathcal{L} } 2 - \delta }
	X
	\frac { \partial^2 X } { \partial s_1 \partial s_2 }
	\,
	d s_1 d s_2
\end{align*}
by the integration by parts.
From the symmetry and integration by parts again,
we obtain
\begin{align*}
	I_3 ( \epsilon , \delta )
	= & \
	4 \iint_{ \epsilon \leqq | s_1 - s_2 | \leqq \frac { \mathcal{L} } 2 - \delta }
	\frac 1 { \Delta s }
	\frac { \partial X } { \partial s_1 }
	\, d s_1 d s_2
	\\
	= & \
	4
	\int_{ ( \mathbb{R} / \mathcal{L} \mathbb{Z} }
	\left(
	\left[
	\frac X { \Delta s }
	\right]_{ s_1 = s_2 - \frac { \mathcal{L} } 2 + \delta }^{ s_1 = s_2 - \epsilon }
	+
	\left[
	\frac X { \Delta s }
	\right]_{ s_1 = s_2 + \epsilon }^{ s_1 = s_2 + \frac { \mathcal{L} } 2 - \delta }
	\right)
	d s_2
	\\
	& \quad
	+ \,
	4
	\iint_{ \epsilon \leqq | s_1 - s_2 | \leqq \frac { \mathcal{L} } 2 - \delta }
	\frac X { | \Delta s |^2 }
	\, d s_1 d s_2
	\\
	= & \
	-
	4
	\int_{ ( \mathbb{R} / \mathcal{L} \mathbb{Z} }
	\left(
	\left[
	\frac X { \Delta s }
	\right]_{ s_1 = s_2 - \epsilon }^{ s_1 = s_2 + \epsilon }
	-
	\left[
	\frac X { \Delta s }
	\right]_{ s_1 = s_2 - \frac { \mathcal{L} } 2 + \delta }^{ s_1 = s_2 + \frac { \mathcal{L} } 2 - \delta }
	\right)
	 d s_2
	\\
	& \quad
	+ \,
	4
	\iint_{ \epsilon \leqq | s_1 - s_2 | \leqq \frac { \mathcal{L} } 2 - \delta }
	\frac X { | \Delta s |^2 }
	\, d s_1 d s_2
	.
\end{align*}
Setting
\[
	Y = \frac { \partial X } { \partial s_2 } + \frac 4 { \Delta s }
	,
\]
we have
\begin{align*}
	&
	I_2 ( \epsilon , \delta )
	+
	I_3 ( \epsilon , \delta )
	\\
	& \quad
	=
	-
	\int_{ \mathbb{R} / \mathcal{L} \mathbb(Z) }
	\left(
	\left[ X \left( \frac { \partial X } { \partial s_2 } + \frac 4 { \Delta s } \right) \right]_{ s_1 = s_2 - \epsilon }^{ s_1 = s_2 + \epsilon }
	-
	\left[ X \left( \frac { \partial X } { \partial s_2 } + \frac 4 { \Delta s } \right) \right]_{ s_1 = s_2 - \frac { \mathcal{L} } 2 + \delta }^{ s_1 = s_2 + \frac { \mathcal{L} } 2  - \delta }
	\right)
	d s_2
	\\
	& \quad \qquad
	- \,
	\iint_{ \epsilon \leqq | s_1 - s_2 | \leqq \frac { \mathcal{L} } 2 - \delta }
	X \left(
	\frac { \partial^2 X } { \partial s_1 \partial s_2 }
	-
	\frac 4 { | \Delta s |^2 }
	\right)
	d s_1 d s_2
	\\
	& \quad
	=
	-
	\int_{ \mathbb{R} / \mathcal{L} \mathbb(Z) }
	\left(
	[ XY ]_{ s_1 = s_2 - \epsilon }^{ s_1 = s_2 + \epsilon }
	-
	[ XY ]_{ s_1 = s_2 - \frac { \mathcal{L} } 2 + \delta }^{ s_1 = s_2 + \frac { \mathcal{L} } 2  - \delta }
	\right)
	d s_2
	\\
	& \quad \qquad
	+ \,
	2 \iint_{ \epsilon \leqq | s_1 - s_2 | \leqq \frac { \mathcal{L} } 2 - \delta }
	X \left(
	\mathscr{M} ( \vecf ) - \mathscr{M}_0 ( \vecf ) + \frac 2 { | \Delta s |^2 }
	\right)
	d s_1 d s_2
	.
\end{align*}
For the last equality,
we have used \pref{M}.
Setting
\[
	Z = XY
\]
and
\[
	J( a )
	=
	\int_{ \mathbb{R} / \mathcal{L} \mathbb{Z} }
	[ Z ]_{ s_1 = s_2 - a }^{ s_1 = s_2 + a }
	d s_2
	,
\]
we have
\begin{align*}
	J (a)
	= & \,
	\int_{ \mathbb{R} / \mathcal{L} \mathbb{Z} }
	\left( 
	Z ( s_2 + a , s_2 ) - Z ( s_2 - a , s_2 )
	\right)
	d s_2
	\\
	= & \,
	\int_{ \mathbb{R} / \mathcal{L} \mathbb{Z} }
	\left( 
	Z ( s_2 + a , s_2 ) - Z ( s_2 , s_2 + a )
	\right)
	d s_2
	.
\end{align*}
Because $ X( s_1 , s_2 ) = X( s_2 , s_1 ) $,
we have
\begin{align*}
	&
	Z ( s_2 + a , s_2 ) - Z ( s_2 , s_2 + a )
	\\
	& \quad
	=
	X( s_2 + a , s_2 ) Y( s_2 + a , s_2 )
	-
	X( s_2 , s_2 + a ) Y( s_2 , s_2 + a )
	\\
	& \quad
	=
	X( s_2 + a , s_2 )
	( Y( s_2 + a , s_2 ) - Y( s_2 , s_2 + a ) )
	.
\end{align*}
Since
\[
	Y = \frac { \partial X } { \partial s_2 } + \frac 4 { \Delta s }
	=
	\frac { 2 \vectau ( s_2 ) \cdot \Delta \vecf } { \| \Delta \vecf \|^2 } + \frac 2 { \Delta s }
	,
\]
it holds that
\begin{align*}
	&
	Y( s_2 + a , s_2 ) - Y( s_2 , s_2 + a )
	\\
	& \quad
	=
	\frac { 2 \vectau ( s_2 + a ) \cdot ( \vecf ( s_2 + a ) - \vecf ( s_2 ) ) } 
	{ \| \vecf ( s_2 + a ) - \vecf ( s_2 ) \|^2 }
	+
	\frac 2 { s_2 + a - s_2 }
	\\
	& \quad \qquad
	- \,
	\frac { 2 \vectau ( s_2 ) \cdot ( \vecf ( s_2 ) - \vecf ( s_2 + a ) ) }
	{ \| \vecf ( s_2 ) - \vecf ( s_2 + a ) \|^2 }
	-
	\frac 2 { s_2 - ( s_2 + a ) }
	\\
	& \quad
	=
	\frac { 2 ( \vectau ( s_2 + a ) + \vectau ( s_2 ) ) \cdot ( \vecf ( s_2 + a ) - \vecf ( s_2 ) ) } 
	{ \| \vecf ( s_2 + a ) - \vecf ( s_2 ) \|^2 }
	+
	\frac 4a
	\\
	& \quad
	=
	\frac { 2 ( \vectau ( s_2 + a ) - \vectau ( s_2 ) ) \cdot ( \vecf ( s_2 + a ) - \vecf ( s_2 ) ) }
	{ \| \vecf ( s_2 + a ) - \vecf ( s_2 ) \|^2 }
	+
	\frac { 4 \vectau ( s_2 ) \cdot ( \vecf ( s_2 + a ) - \vecf ( s_2 ) ) }
	{ \| \vecf ( s_2 + a ) - \vecf ( s_2 ) \|^2 }
	+
	\frac 4a
	\\
	& \quad
	=
	2 \frac d { d s_2 } \log \| \vecf ( s_2 + a ) - \vecf \|^2 
	+
	\frac { 4 \vectau ( s_2 ) \cdot ( \vecf ( s_2 + a ) - \vecf ( s_2 ) ) }
	{ \| \vecf ( s_2 + a ) - \vecf ( s_2 ) \|^2 }
	+
	\frac 4a
	\\
	& \quad
	=
	-
	2 \frac d { d s_2 } X ( s_2 + a , s_2 )
	+
	\frac { 4 \vectau ( s_2 ) \cdot ( \vecf ( s_2 + a ) - \vecf ( s_2 ) ) }
	{ \| \vecf ( s_2 + a ) - \vecf ( s_2 ) \|^2 }
	+
	\frac 4a
	.
\end{align*}
Hence,
we obtain
\begin{align*}
	J(a)
	= & \
	J_1 (a) + J_2 (a) + J_3 (a)
	,
	\\
	J_1(a)
	= & \
	-
	\int_{ \mathbb{R} / \mathcal{L} \mathbb{Z} }
	\frac d { d s_2 } X( s_2 + a , s_2 )^2 d s_2
	,
	\\
	J_2(a)
	= & \
	\int_{ \mathbb{R} / \mathcal{L} \mathbb{Z} }
	\frac { 4 X( s_2 + a , s_2 ) \vectau ( s_2 ) \cdot ( \vecf ( s_2 + a ) - \vecf ( s_2 ) ) } { \| \vecf ( s_2 + a ) - \vecf ( s_2 ) \|^2 }
	\,
	d s_2
	,
	\\
	J_3(a)
	= & \
	\frac 4a \int_{ \mathbb{R} / \mathcal{L} \mathbb{Z} } X( s_2 + a , s_2 ) \, d s_2
	.
\end{align*}
By the periodicity of $ \vecf $,
we have $ J_1 (a) = 0 $.
\par
We can show
\[
	\lim_{ \epsilon \to + 0 } J_2 ( \epsilon )
	=
	\lim_{ \delta \to +0 } J_2 \left( a + \frac { \mathcal{L} } 2 - \delta \right)
	=
	0
	;
\]
however,
the proof is different for the cases of $ \epsilon \to + 0 $ and $ \delta \to + 0 $.
\par
Set $ a = \epsilon $,
and take the limit as $ \epsilon \to + 0 $.
It follows from $ \log x \leqq x - 1 $ that 
\[
	0
	\leqq
	X( s_1 , s_2 )
	\leqq
	\frac { | \Delta s |^2 } { \| \Delta \vecf \|^2 } - 1
	.
\]
Therefore we have
\begin{align*}
	| J_2 ( \epsilon ) |
	\leqq & \
	\int_{ \mathbb{R} / \mathcal{L} \mathbb{Z} }
	\frac { 4 X( s_2 + \epsilon , s_2 ) } { \| \vecf ( s_2 + \epsilon ) - \vecf ( s_2 ) \| }
	\,
	d s_2
	\\
	\leqq & \
	\int_{ \mathbb{R} / \mathcal{L} \mathbb{Z} }
	\frac 4 { \| \vecf ( s_2 + \epsilon ) - \vecf ( s_2 ) \| }
	\left( \frac { \epsilon^2 } { \| \vecf ( s_2 + \epsilon ) - \vecf ( s_2 ) \|^2 } - 1 \right)
	\\
	= & \
	\int_{ \mathbb{R} / \mathcal{L} \mathbb{Z} }
	\frac 2 { \| \vecf ( s_2 + \epsilon ) - \vecf ( s_2 ) \|^3 }
	\int_{ s_2 }^{ s_2 + \epsilon }
	\int_{ s_2 }^{ s_2 + \epsilon }
	\| \vectau ( s_3 ) - \vectau ( s_4 ) \|^2 d s_3 d s_4
. 
\end{align*}
From $ \mathcal{E}_0 ( \vecf ) < \infty $,
the function $ \vecf $ satisfies the bi-Lipschitz estimate.
Combining this and the argument in the proof of \cite[Theorem 2.1]{IshizekiNagasawaI},
we obtain
\begin{align*}
	| J_2 ( \epsilon ) |
	\leqq & \
	\int_{ \mathbb{R} / \mathcal{L} \mathbb{Z} }
	\frac C { \epsilon^3 }
	\int_{ s_2 }^{ s_2 + \epsilon }
	\int_{ s_2 }^{ s_2 + \epsilon }
	\| \vectau ( s_3 ) - \vectau ( s_4 ) \|^2 d s_3 d s_4
	\\
	\leqq & \
	C
	\int_{ \mathbb{R} / \mathcal{L} \mathbb{Z} }
	\int_{ s_4 - \epsilon }^{ s_4 + \epsilon }
	\frac { \| \vectau ( s_3 ) - \vectau ( s_4 ) \|^2 }
	{ ( s_3 - s_4 )^2 }
	\,
	d s_3 d s_4
	.
\end{align*}
$ \mathcal{E}_0 ( \vecf ) < \infty $ also implies $ \vecf \in W^{ \frac 12 , 2 } ( \mathbb{R} / \mathcal{L} \mathbb{Z} ) $.
The absolute continuity of the integral gives us
\[
	\lim_{ \epsilon \to + 0 } | J_2 ( \epsilon ) | = 0 .
\]
\par
Next,
set $ a = \frac { \mathcal{L} } 2 - \delta $ and let take the limit as $ \delta \to + 0 $. 
When $ \delta > 0 $ is small,
the integrand of $ J_2 \left( \frac { \mathcal{L} } 2 - \delta \right) $ is uniformly bounded in both $ \delta $ and $ s_2 $.
Therefore we can apply Lebesgue's convergence theorem to see that
\[
	\lim_{ \delta \to + 0 } J_2 \left( \frac { \mathcal{L} } 2 - \delta \right)
	=
	J_2 \left( \frac { \mathcal{L} } 2 \right)
	.
\]
Using the periodicity of $ \vecf $,
an appropriate change of variables,
and $ X ( s_1 , s_2 ) = X( s_2 , s_1 ) $,
we have
\begin{align}
	J_2 \left( \frac { \mathcal{L} } 2 \right)
	= & \
	\int_{ \mathbb{R} / \mathcal{L} \mathbb{Z} }
	\frac { 4 X\left( s_2 + \frac { \mathcal{L} } 2 , s_2 \right) \vectau ( s_2 ) \cdot \left( \vecf \left( s_2 + \frac { \mathcal{L} } 2 \right) - \vecf ( s_2 ) \right) }
	{ \left\| \vecf \left( s_2 + \frac { \mathcal{L} } 2 \right) - \vecf ( s_2 ) \right\|^2 }
	\,
	d s_2
	\label{J21}
	\\
	= & \
	\int_{ \mathbb{R} / \mathcal{L} \mathbb{Z} }
	\frac { 4 X\left( s_2 - \frac { \mathcal{L} } 2 , s_2 \right)
	\vectau ( s_2 ) \cdot \left( \vecf \left( s_2 - \frac { \mathcal{L} } 2 \right) - \vecf ( s_2 ) \right) }
	{ \left\| \vecf \left( s_2 - \frac { \mathcal{L} } 2 \right) - \vecf ( s_2 ) \right\|^2 }
	\,
	d s_2
	\nonumber
	\\
	= & \
	\int_{ \mathbb{R} / \mathcal{L} \mathbb{Z} }
	\frac { 4 X\left( s_2 , s_2 + \frac { \mathcal{L} } 2 \right) \vectau \left( s_2 + \frac { \mathcal{L} } 2 \right)
	\cdot \left( \vecf ( s_2 ) - \vecf \left( s_2 + \frac { \mathcal{L} } 2 \right) \right) }
	{ \left\| \vecf ( s_2 ) - \vecf \left( s_2 +\frac { \mathcal{L} } 2 \right) \right\|^2 }
	\,
	d s_2
	\nonumber
	\\
	= & \
	\int_{ \mathbb{R} / \mathcal{L} \mathbb{Z} }
	\frac { 4 X\left( s_2 + \frac { \mathcal{L} } 2 , s_2 \right) \vectau \left( s_2 + \frac { \mathcal{L} } 2 \right)
	\cdot \left( \vecf ( s_2 ) - \vecf \left( s_2 + \frac { \mathcal{L} } 2 \right) \right) }
	{ \left\| \vecf ( s_2 ) - \vecf \left( s_2 +\frac { \mathcal{L} } 2 \right) \right\|^2 }
	\,
	d s_2
	.
	\label{J22}
\end{align}
Hence,
taking the average of \pref{J21} and \pref{J22},
we obtain
\begin{align*}
	J_2 \left( \frac { \mathcal{L} } 2 \right)
	= & \
	-
	\int_{ \mathbb{R} / \mathcal{L} \mathbb{Z} }
	\frac { 2 X\left( s_2 + \frac { \mathcal{L} } 2 , s_2 \right)
	\left( \vectau \left( s_2 + \frac { \mathcal{L} } 2 \right) - \vectau ( s_2 ) \right)
	\cdot
	\left( \vecf \left( s_2 + \frac { \mathcal{L} } 2 \right) - \vecf ( s_2 ) \right) }
	{ \left\| \vecf \left( s_2 +\frac { \mathcal{L} } 2 \right) - \vecf ( s_2 ) \right\|^2 }
	\,
	d s_2
	\\
	= & \
	-
	\int_{ \mathbb{R} / \mathcal{L} \mathbb{Z} }
	X\left( s_2 + \frac { \mathcal{L} } 2 , s_2 \right)
	\frac d { d s_2 } \log \left\| \vecf \left( s_2 +\frac { \mathcal{L} } 2 \right) - \vecf ( s_2 ) \right\|^2
	d s_2
	\\
	= & \
	\frac 12
	\int_{ \mathbb{R} / \mathcal{L} \mathbb{Z} }
	\frac d { d s_2 } 
	X\left( s_2 + \frac { \mathcal{L} } 2 , s_2 \right)^2
	d s_2
	\\
	= & \
	0
	.
\end{align*}
\par
Set $ a = \epsilon $.
It follows from $ \log x \leqq x - 1 $ that
\begin{align*}
	\frac 4 \epsilon X ( s_2 + \epsilon , s_2 )
	\leqq & \
	\frac 4 \epsilon
	\left( \frac { \epsilon^2 } { \| \vecf ( s_2 + \epsilon ) - \vecf ( s_2 ) \|^2 } - 1 \right)
	\\
	= & \
	\frac 2 { \epsilon \| \vecf ( s_2 + \epsilon ) - \vecf ( s_2 ) \|^2 }
	\int_{ s_2 }^{ s_2 + \epsilon }
	\int_{ s_2 }^{ s_2 + \epsilon }
	\| \vectau ( s_3 ) - \vectau ( s_4 ) \|^2
	d s_3 d s_4
	.
\end{align*}
Hence,
we have
\[
	| J_3 ( \epsilon ) |
	\leqq
	C \int_{ \mathbb{R} / \mathcal{L} \mathbb{Z} }
	\int_{ s_4 - \epsilon }^{ s_4 + \epsilon }
	\frac { \| \vectau ( s_3 ) - \vectau ( s_4 ) \|^2 } { ( s_3 - s_4 )^2 }
	\, d s_3 d s_4
	\to
	0 \quad ( \epsilon \to + 0 )
	.
\]
When $ \delta \to + 0 $,
we apply the Lebesgue convergence theorem and have
\[
	J_3 \left( \frac { \mathcal{L} } 2 - \delta \right)
	\to
	J_3 \left( \frac { \mathcal{L} } 2 \right)
	=
	\frac 8 { \mathcal{L} }
	\int_{ \mathbb{R} / \mathcal{L} \mathbb{Z} }
	X \left( s + \frac { \mathcal{L} } 2 , s \right) ds
	.
\]
This integration is absolutely convergent by $ 0 \leqq X \leqq \| X \|_{ L^\infty } < \infty $.
Similarly,
\begin{gather*}
	\iint_{ ( \mathbb{R} / \mathcal{L} \mathbb{Z} )^2 }
	X
	\mathscr{M} ( \vecf )
	\,
	d s_1 d s_2
	,
	\quad
	\iint_{ ( \mathbb{R} / \mathcal{L} \mathbb{Z} )^2 }
	X
	\mathscr{M}_0 ( \vecf )
	\,
	d s_1 d s_2
\end{gather*}
is also absolutely convergent.
Moreover,
since
\[
	0 \leqq \frac X { | \Delta s |^2 }
	\leqq
	\frac 1 { | \Delta s |^2 }
	\left\{ \frac { | \Delta s |^2 } { \| \Delta \vecf \|^2 } - 1 \right\}
	=
	\frac 1 { \| \Delta \vecf \|^2 }
	-
	\frac 1 { | \Delta s |^2 }
	=
	\mathscr{M} ( \vecf )
	,
\]
we find that
\[
	\iint_{ ( \mathbb{R} / \mathcal{L} \mathbb{Z} )^2 }
	\frac X { | \Delta s |^2 }
	\,
	d s_1 d s_2
\]
is also absolutely convergent.
\qed
\end{proof}
\begin{thm}
If $ \mathcal{E}_0 ( \vecf ) < \infty $,
then it holds that
\begin{align*}
	0 \leqq
	\mathcal{E}_1 ( \vecf )
	\leqq & \
	\left( 3 + \| X \|_{ L^\infty } \right)
	\mathcal{E}_0 ( \vecf )
	+
	4 \left( 4 + \| X \|_{ L^\infty } \right)
	,
	\\
	&
	\hspace{-75pt}
	-
	\left( 2 + \| X \|_{ L^\infty } \right) \mathcal{E}_0 ( \vecf )
	- 4 \left( 4 + \| X \|_{ L^\infty } \right)
	\\
	\leqq
	\mathcal{E}_2 ( \vecf )
	\leqq & \
	\min \left\{
	\| X \|_{ L^\infty } \left( \mathcal{E}_0 ( \vecf ) + 4 \right) - 8
	,
	\mathcal{E}_0 ( \vecf )
	\right\}
	.
\end{align*}
\label{upper and lower bounds}
\end{thm}
\begin{proof}
Noticing $ X \geqq 0 $,
we estimate $ \mathcal{E}_2 ( \vecf ) $ from above and below by splitting the integrand of \pref{E2} into positive and negative parts.
Also,
we use
\[
	\frac X { | \Delta s |^2 }
	\leqq
	\mathscr{M} ( \vecf )
	,
	\quad
	\mathcal{E} ( \vecf ) = \mathcal{E}_0 ( \vecf ) + 4
	.
\]
As results,
we obtain
\begin{align*}
	\mathcal{E}_2 ( \vecf )
	\leqq & \
	\iint_{ ( \mathbb{R} / \mathcal{L} \mathbb{Z} )^2 }
	X \mathscr{M}_0 ( \vecf )
	\, d s_1 d s_2
	+
	\frac 4 { \mathcal{L} }
	\int_{ \mathbb{R} / \mathcal{L} \mathbb{Z} }
	X \left( s + \frac { \mathcal{L} } 2 , s \right) ds
	- 8
	\\
	\leqq & \
	\| X \|_{ L^\infty } \left( \mathcal{E}_0 ( \vecf ) + 4 \right)
	- 8
	,
	\\
	\mathcal{E}_2 ( \vecf )
	\geqq & \
	-
	\iint_{ ( \mathbb{R} / \mathcal{L} \mathbb{Z} )^2 }
	X
	\left(
	\mathscr{M} ( \vecf )
	+
	\frac 2 { | \Delta s |^2 }
	\right)
	d s_1 d s_2
	- 8
	\\
	\geqq & \
	-
	\| X \|_{ L^\infty } \mathcal{E} ( \vecf )
	-
	2 \mathcal{E} ( \vecf )
	-
	8
	\\
	= & \
	-
	\left( 2 + \| X \|_{ L^\infty } \right) \mathcal{E}_0 ( \vecf )
	- 4 \left( 4 + \| X \|_{ L^\infty } \right)
	.
\end{align*}
Combining this,
\pref{E0=E1+E2} and the non-negativity of $ \mathcal{E}_1 ( \vecf ) $,
we have
\begin{align*}
	\mathcal{E}_1 ( \vecf )
	=
	& \
	\mathcal{E}_0 ( \vecf )
	-
	\mathcal{E}_2 ( \vecf )
	\leqq
	\left( 3 + \| X \|_{ L^\infty } \right)
	\mathcal{E}_0 ( \vecf )
	+
	4 \left( 4 + \| X \|_{ L^\infty } \right)
	,
	\\
	\mathcal{E}_2 ( \vecf )
	\leqq
	& \
	\mathcal{E}_0 ( \vecf )
	.
\end{align*}
\qed
\end{proof}
\begin{rem}
In \cite{IshizekiNagasawaIII},
a non-trivial lower bound
\[
	\mathcal{E}_1 ( \vecf )
	\geqq
	2 \pi^2
\]
was given under the asssumption $ \vecf \in C^{1,1} ( \mathbb{R} / \mathcal{L} \mathbb{Z} ) $.
\end{rem}
\section{Modulus of continuity}
\par
Let $ \mathrm{Im} \widetilde { \vecf } $ be an embedded closed curve other than $ \mathrm{Im} \vecf $.
In this section,
we estimate $ \left| \mathcal{E}_i ( \vecf ) - \mathcal{E}_i ( \widetilde { \vecf } ) \right| $ ($ i = 1 $,
$ 2 $) in terms of certain quantities which vanish when $ \vecf = \widetilde { \vecf  }$.
Since the energy $ \mathcal{E} $ is scaling invariant,
we may assume that the total length of $ \mathrm{Im} \widetilde { \vecf } $ is the same as that of $ \mathrm{Im} \vecf $.
Set
\[
	\widetilde { \vectau } = \widetilde { \vecf^\prime } ,
	\quad
	\widetilde X = \log \frac { | \Delta s |^2 } { \| \Delta \widetilde { \vecf } \|^2 } .
\]
\begin{prop}
Assume that $ \vecf $ and $ \widetilde { \vecf } $ satisfy $ \mathcal{E}_0 ( \vecf ) < \infty $,
$ \mathcal{E}_0 ( \widetilde { \vecf } ) < \infty $,
and that they have the same total length $ \mathcal{L} $.
Then it holds that
\begin{align*}
	&
	\mathcal{E}_1 ( \vecf )
	-
	\mathcal{E}_1 ( \widetilde { \vecf } )
	\\
	& \quad
	=
	\mathcal{E}_0 ( \vecf )
	-
	\mathcal{E}_0 ( \widetilde { \vecf } )
	\\
	& \quad \qquad
	+ \,
	\iint_{ ( \mathbb{R} / \mathcal{L} \mathbb{Z} )^2 }
	( X - \widetilde X )
	\left(
	\mathscr{M} ( \vecf ) - \mathscr{M}_0 ( \vecf )
	+
	\mathscr{M} ( \widetilde { \vecf } ) - \mathscr{M}_0 ( \widetilde { \vecf } )
	+
	\frac 2 { | \Delta s |^2 }
	\right)
	d s_1 d s_2
	,
	\\
	&
	\mathcal{E}_2 ( \vecf )
	-
	\mathcal{E}_2 ( \widetilde { \vecf } )
	\\
	& \quad
	=
	-
	\iint_{ ( \mathbb{R} / \mathcal{L} \mathbb{Z} )^2 }
	( X - \widetilde X )
	\left(
	\mathscr{M} ( \vecf ) - \mathscr{M}_0 ( \vecf )
	+
	\mathscr{M} ( \widetilde { \vecf } ) - \mathscr{M}_0 ( \widetilde { \vecf } )
	+
	\frac 2 { | \Delta s |^2 }
	\right)
	d s_1 d s_2
	.
\end{align*}
\label{M_f-M_tilde f}
\end{prop}
\begin{proof}
By \pref{E0=E1+E2},
it is enough to show the assertion on $ \mathcal{E}_2 $.
We have already seen
\[
	\mathscr{M}_2 ( \vecf )
	=
	-
	2
	\frac { \partial^2 } { \partial s_1 \partial s_2 } \log \| \Delta \vecf \|^2
	-
	\left( \frac \partial { \partial s_1 } \log \| \Delta \vecf \|^2 \right)
	\left( \frac \partial { \partial s_2 } \log \| \Delta \vecf \|^2 \right)
\]
for $ \displaystyle{ 0 < | s_1 - s_2 | < \frac { \mathcal{L} } 2 } $ in the proof of Lemma \ref{M_f}.
Therefore it holds that
\begin{align*}
	&
	2 \left(
	\mathscr{M}_2 ( \vecf )
	-
	\mathscr{M}_2 ( \widetilde { \vecf } )
	\right)
	\\
	& \quad
	=
	-
	2
	\frac { \partial^2 } { \partial s_1 \partial s_2 } \log \| \Delta \vecf \|^2
	-
	\left( \frac \partial { \partial s_1 } \log \| \Delta \vecf \|^2 \right)
	\left( \frac \partial { \partial s_2 } \log \| \Delta \vecf \|^2 \right)
	\\
	& \quad \qquad
	+ \,
	2
	\frac { \partial^2 } { \partial s_1 \partial s_2 } \log \| \Delta \widetilde { \vecf } \|^2
	+
	\left( \frac \partial { \partial s_1 } \log \| \Delta \widetilde { \vecf } \|^2 \right)
	\left( \frac \partial { \partial s_2 } \log \| \Delta \widetilde { \vecf } \|^2 \right)
	\\
	& \quad
	=
	2
	\frac { \partial^2 } { \partial s_1 \partial s_2 } \log \frac { \| \Delta \widetilde { \vecf } \|^2 } { \| \Delta \vecf \|^2 }
	\\
	& \quad \qquad
	- \,
	\frac 12
	\left\{
	\frac \partial { \partial s_1 }
	\left(
	\log \| \Delta \vecf \|^2
	-
	\log \| \Delta \widetilde{ \vecf } \|^2
	\right)
	\right\}
	\left\{
	\frac \partial { \partial s_2 }
	\left(
	\log \| \Delta \vecf \|^2
	+
	\log \| \Delta \widetilde{ \vecf } \|^2
	\right)
	\right\}
	\\
	& \quad \qquad
	- \,
	\frac 12
	\left\{
	\frac \partial { \partial s_1 }
	\left(
	\log \| \Delta \vecf \|^2
	+
	\log \| \Delta \widetilde{ \vecf } \|^2
	\right)
	\right\}
	\left\{
	\frac \partial { \partial s_2 }
	\left(
	\log \| \Delta \vecf \|^2
	-
	\log \| \Delta \widetilde{ \vecf } \|^2
	\right)
	\right\}
	\\
	& \quad
	=
	2
	\frac { \partial^2 } { \partial s_1 \partial s_2 } ( X - \widetilde X )
	-
	\frac 12
	\left\{ \frac \partial { \partial s_1 } ( X - \widetilde X ) \right\}
	\left\{
	\frac \partial { \partial s_2 }
	\left( X + \widetilde X - 2 \log | \Delta s |^2 \right)
	\right\}
	\\
	& \quad \qquad
	- \,
	\frac 12
	\left\{
	\frac \partial { \partial s_1 }
	\left( X + \widetilde X - 2 \log | \Delta s |^2 \right\}
	\right\}
	\left\{ \frac \partial { \partial s_2 } ( X - \widetilde X ) \right\}
	.
\end{align*}
In a manner similar to that used for the proof of \pref{lim I_1},
we can derive
\[
	\lim_{ \substack{ \epsilon \to + 0 \\ \delta \to + 0 } }
	\iint_{ \epsilon \leqq | s_1 - s_2 | \leqq \frac { \mathcal{L} } 2 - \delta }
	2 \frac { \partial^2 } { \partial s_1 \partial s_2 } ( X - \widetilde X )
	\, d s_1 d s_2
	=
	- 8 + 8
	= 0
	.
\]
By using the symmetry of the integrand with respect to $ s_1 $ and $ s_2 $,
we have
\begin{align*}
	&
	\iint_{ \epsilon \leqq | s_1 - s_2 | \leqq \frac { \mathcal{L} } 2 - \delta }
	\left[
	-
	\frac 12
	\left\{ \frac \partial { \partial s_1 } ( X - \widetilde X ) \right\}
	\left\{
	\frac \partial { \partial s_2 }
	\left( X + \widetilde X - 2 \log | \Delta s |^2 \right)
	\right\}
	\right.
	\\
	& \quad \qquad \qquad \qquad
	\left.
	- \,
	\frac 12
	\left\{
	\frac \partial { \partial s_1 }
	\left( X + \widetilde X - 2 \log | \Delta s |^2 \right)
	\right\}
	\left\{ \frac \partial { \partial s_2 } ( X - \widetilde X ) \right\}
	\right]
	d s_1 d s_2
	\\
	& \quad
	=
	-
	\iint_{ \epsilon \leqq | s_1 - s_2 | \leqq \frac { \mathcal{L} } 2 - \delta }
	\left\{ \frac \partial { \partial s_1 } ( X - \widetilde X ) \right\}
	\left\{
	\frac \partial { \partial s_2 }
	\left( X + \widetilde X - 2 \log | \Delta s |^2 \right)
	\right\}
	d s_1 d s_2
	.
\end{align*}
Consequently we obtain
\begin{align*}
	&
	2 \left(
	\mathcal{E}_2 ( \vecf )
	-
	\mathcal{E}_2 ( \widetilde { \vecf } )
	\right)
	\\
	& \quad
	=
	-
	\lim_{ \substack{ \epsilon \to + 0 \\ \delta \to + 0 } }
	\iint_{ \epsilon \leqq | s_1 - s_2 | \leqq \frac { \mathcal{L} } 2 - \delta }
	\left\{
	\frac \partial { \partial s_1 } ( X - \widetilde X )
	\right\}
	\left\{
	\frac \partial { \partial s_2 } \left( X + \widetilde X - 2 \log | \Delta s |^2 \right)
	\right\}
	d s_1 d s_2
	.
\end{align*}
Setting
\[
	K(a)
	=
	\int_{ \mathbb{R} / \mathcal{L} \mathbb{Z} }
	\left[
	( X - \widetilde X )
	\frac \partial { \partial s_2 } \left( X + \widetilde X - 2 \log | \Delta s |^2 \right)
	\right]_{ s_1 = s_2 - a }^{ s_1 = s_1 + a }
	d s_2
	,
\]
we have
\begin{align*}
	&
	-
	\iint_{ \epsilon \leqq | s_1 - s_2 | \leqq \frac { \mathcal{L} } 2 - \delta }
	\left\{
	\frac \partial { \partial s_1 } ( X - \widetilde X )
	\right\}
	\left\{
	\frac \partial { \partial s_2 } \left( X + \widetilde X - 2 \log | \Delta s |^2 \right)
	\right\}
	d s_1 d s_2
	\\
	& \quad
	=
	K( \epsilon )
	-
	K \left( \frac { \mathcal{L} } 2 - \delta \right)
	\\
	& \quad \qquad
	+ \,
	\iint_{ \epsilon \leqq | s_1 - s_2 | \leqq \frac { \mathcal{L} } 2 - \delta }
	( X - \widetilde X )
	\frac { \partial^2 } { \partial s_1 \partial s_2 } \left( X + \widetilde X - 2 \log | \Delta s |^2 \right)
	d s_1 d s_2
	.
\end{align*}
\par
Now,
we set
\[
	\bar X = X - \widetilde X
	,
	\quad
	\bar Y =
	\frac \partial { \partial s_2 } \left( X + \widetilde X - 2 \log | \Delta s |^2 \right)
	.
\]
Since $ \bar X ( s_1 , s_2 ) = \bar X ( s_2 , s_1 ) $,
it holds that
\[
	K(a)
	=
	\int_{ \mathbb{R} / \mathcal{L} \mathbb{Z} }
	\bar X( s_2 + a , s_2 )
	( \bar Y ( s_2 + a , s_2 ) - \bar Y( s_2 , s_2 + a ) )
	\, d s_2
	.
\]
Set
\[
	\bar Y
	=
	\frac { 2 \vectau ( s_2 ) \cdot \Delta \vecf } { \| \Delta \vecf \|^2 }
	+
	\frac { 2 \widetilde{ \vectau } ( s_2 ) \cdot \Delta \widetilde { \vecf } } { \| \Delta \widetilde { \vecf } \|^2 }
	,
\]
and then we have
\begin{align*}
	&
	\bar Y ( s_2 + a , s_2 ) - \bar Y( s_2 , s_2 + a )
	\\
	& \quad
	=
	\frac { 2 \vectau ( s_2 + a ) \cdot ( \vecf ( s_2 + a ) - \vecf ( s_2 ) ) }
	{ \| \vecf ( s_2 + a ) - \vecf ( s_2 ) \|^2 }
	+
	\frac { 2 \widetilde{ \vectau } ( s_2 + a ) \cdot ( \widetilde { \vecf } ( s_2 + a ) - \widetilde { \vecf } ( s_2 ) ) }
	{ \| \widetilde { \vecf } ( s_2 + a ) - \widetilde { \vecf } ( s_2 ) \|^2 }
	\\
	& \quad \qquad
	- \,
	\frac { 2 \vectau ( s_2 ) \cdot ( \vecf ( s_2 ) - \vecf ( s_2 + a ) ) }
	{ \| \vecf ( s_2 ) - \vecf ( s_2 + a ) \|^2 }
	-
	\frac { 2 \widetilde{ \vectau } ( s_2 ) \cdot ( \widetilde { \vecf } ( s_2 ) - \widetilde { \vecf } ( s_2 + a ) ) }
	{ \| \widetilde { \vecf } ( s_2 ) - \widetilde { \vecf } ( s_2 + a ) \|^2 }
	\\
	& \quad
	=
	\frac { 2 ( \vectau ( s_2 + a ) + \vectau ( s_2 ) ) \cdot ( \vecf ( s_2 + a ) - \vecf ( s_2 ) ) }
	{ \| \vecf ( s_2 + a ) - \vecf ( s_2 ) \|^2 }
	\\
	& \quad \qquad
	+ \,
	\frac { 2 ( \widetilde{ \vectau } ( s_2 + a ) + \widetilde { \vectau } ( s_2 ) ) \cdot ( \widetilde { \vecf } ( s_2 + a ) - \widetilde { \vecf } ( s_2 ) ) }
	{ \| \widetilde { \vecf } ( s_2 + a ) - \widetilde { \vecf } ( s_2 ) \|^2 }
	\\
	& \quad
	=
	\frac { 2 ( \vectau ( s_2 + a ) - \vectau ( s_2 ) ) \cdot ( \vecf ( s_2 + a ) - \vecf ( s_2 ) ) }
	{ \| \vecf ( s_2 + a ) - \vecf ( s_2 ) \|^2 }
	+
	\frac { 4 \vectau ( s_2 ) \cdot ( \vecf ( s_2 + a ) - \vecf ( s_2 ) ) }
	{ \| \vecf ( s_2 + a ) - \vecf ( s_2 ) \|^2 }
	\\
	& \quad \qquad
	- \,
	\frac { 2 ( \widetilde{ \vectau } ( s_2 + a ) - \widetilde { \vectau } ( s_2 ) ) \cdot ( \widetilde { \vecf } ( s_2 + a ) - \widetilde { \vecf } ( s_2 ) ) }
	{ \| \widetilde { \vecf } ( s_2 + a ) - \widetilde { \vecf } ( s_2 ) \|^2 }
	+
	\frac { 4 \widetilde{ \vectau } ( s_2 + a ) \cdot ( \widetilde { \vecf } ( s_2 + a ) - \widetilde { \vecf } ( s_2 ) ) }
	{ \| \widetilde { \vecf } ( s_2 + a ) - \widetilde { \vecf } ( s_2 ) \|^2 }
	\\
	& \quad
	=
	\frac d { d s_2 } \log \| \vecf ( s_2 + a ) - \vecf ( s_2 ) \|^2
	+
	\frac { 4 \vectau ( s_2 ) \cdot ( \vecf ( s_2 + a ) - \vecf ( s_2 ) ) }
	{ \| \vecf ( s_2 + a ) - \vecf ( s_2 ) \|^2 }
	\\
	& \quad \qquad
	- \,
	\frac d { d s_2 } \log \| \widetilde { \vecf } ( s_2 + a ) - \widetilde { \vecf } ( s_2 ) \|^2
	+
	\frac { 4 \widetilde{ \vectau } ( s_2 + a ) \cdot ( \widetilde { \vecf } ( s_2 + a ) - \widetilde { \vecf } ( s_2 ) ) }
	{ \| \widetilde { \vecf } ( s_2 + a ) - \widetilde { \vecf } ( s_2 ) \|^2 }
	\\
	& \quad
	=
	\frac d { d s_2 } \bar X
	+
	\frac { 4 \vectau ( s_2 ) \cdot ( \vecf ( s_2 + a ) - \vecf ( s_2 ) ) }
	{ \| \vecf ( s_2 + a ) - \vecf ( s_2 ) \|^2 }
	+
	\frac { 4 \widetilde{ \vectau } ( s_2 + a ) \cdot ( \widetilde { \vecf } ( s_2 + a ) - \widetilde { \vecf } ( s_2 ) ) }
	{ \| \widetilde { \vecf } ( s_2 + a ) - \widetilde { \vecf } ( s_2 ) \|^2 }
	.
\end{align*}
Therefore $ K(a) $ is written as
\begin{align*}
	K(a)
	= & \
	K_1 (a)
	+
	K_2 (a)
	,
	\\
	K_1 (a)
	= & \
	\int_{ \mathbb{R} / \mathcal{L} \mathbb{Z} }
	\frac { 4 \bar X( s_2 + a , s_2 ) \vectau ( s_2 ) \cdot ( \vecf ( s_2 + a ) - \vecf ( s_2 ) ) }
	{ \| \vecf ( s_2 + a ) - \vecf ( s_2 ) \|^2 }
	\, d s_1 d s_2
	,
	\\
	K_2 (a)
	= & \
	\int_{ \mathbb{R} / \mathcal{L} \mathbb{Z} }
	\frac { 4 \bar X( s_2 + a , s_2 ) \widetilde{ \vectau } ( s_2 + a ) \cdot ( \widetilde { \vecf } ( s_2 + a ) - \widetilde { \vecf } ( s_2 ) ) }
	{ \| \widetilde { \vecf } ( s_2 + a ) - \widetilde { \vecf } ( s_2 ) \|^2 }
	\, d s_1 d s_2
	.
\end{align*}
\par
In a similar manner as for the estimate of $ J_2 ( \epsilon ) $,
we can show
\[
	K_1 ( \epsilon ) \to 0,
	\quad
	K_2 ( \epsilon ) \to 0
	\quad ( \epsilon \to + 0 )
	.
\]
\par
Set $ a = \frac { \mathcal{L} } 2 - \delta $,
and take the limit as $ \delta \to + 0 $.
Lebesgue's convergence theorem gives us
\[
	K_1 \left( \frac { \mathcal{L} } 2 - \delta \right)
	\to
	K_1 \left( \frac { \mathcal{L} } 2 \right)
	,
	\quad
	K_2 \left( \frac { \mathcal{L} } 2 - \delta \right)
	\to
	K_2 \left( \frac { \mathcal{L} } 2 \right)
	\quad ( \delta \to + 0 )
	.
\]
In the same way as for the calculation of $ \displaystyle{ J_2 \left( \frac { \mathcal{L} } 2 \right) } $,
we obtain
\[
	K_1 \left( \frac { \mathcal{L} } 2 \right)
	=
	\int_{ \mathbb{R} / \mathcal{L} \mathbb{Z} }
	\bar X \left( s_2 + \frac { \mathcal{L} } 2 , s_2 \right)
	\frac d { d s_2 } X \left( s_2 + \frac { \mathcal{L} } 2 , s_2 \right)
	d s_2
	.
\]
Similarly we have
\begin{align*}
	K_2 \left( \frac { \mathcal{L} } 2 \right)
	= & \
	\int_{ \mathbb{R} / \mathcal{L} \mathbb{Z} }
	\frac { 4 \bar X\left( s_2 + \frac { \mathcal{L} } 2 , s_2 \right)
	\widetilde{ \vectau } \left( s_2 + \frac { \mathcal{L} } 2 \right) \cdot
	\left( \widetilde { \vecf } \left( s_2 + \frac { \mathcal{L} } 2 \right) - \widetilde { \vecf } ( s_2 ) \right) }
	{ \left\| \widetilde { \vecf } \left( s_2 + \frac { \mathcal{L} } 2 \right) - \widetilde { \vecf } ( s_2 ) \right\|^2 }
	\, d s_1 d s_2
	\\
	= & \
	\int_{ \mathbb{R} / \mathcal{L} \mathbb{Z} }
	\frac { 4 \bar X\left( s_2 - \frac { \mathcal{L} } 2 , s_2 \right)
	\widetilde{ \vectau } \left( s_2 - \frac { \mathcal{L} } 2 \right) \cdot
	\left( \widetilde { \vecf } \left( s_2 - \frac { \mathcal{L} } 2 \right) - \widetilde { \vecf } ( s_2 ) \right) }
	{ \left\| \widetilde { \vecf } \left( s_2 - \frac { \mathcal{L} } 2 \right) - \widetilde { \vecf } ( s_2 ) \right\|^2 }
	\, d s_1 d s_2
	\\
	= & \
	\int_{ \mathbb{R} / \mathcal{L} \mathbb{Z} }
	\frac { 4 \bar X\left( s_2 , s_2 + \frac { \mathcal{L} } 2 \right)
	\widetilde{ \vectau } ( s_2 ) \cdot
	\left( \widetilde { \vecf } ( s_2 ) - \widetilde { \vecf } \left( s_2 + \frac { \mathcal{L} } 2 \right) \right) }
	{ \left\| \widetilde { \vecf } ( s_2 ) - \widetilde { \vecf } \left( s_2 + \frac { \mathcal{L} } 2 \right) \right\|^2 }
	\, d s_1 d s_2
	\\
	= & \
	\int_{ \mathbb{R} / \mathcal{L} \mathbb{Z} }
	\frac { 4 \bar X\left( s_2 + \frac { \mathcal{L} } 2 , s_2 \right)
	\widetilde{ \vectau } ( s_2 ) \cdot
	\left( \widetilde { \vecf } ( s_2 ) - \widetilde { \vecf } \left( s_2 + \frac { \mathcal{L} } 2 \right) \right) }
	{ \left\| \widetilde { \vecf } ( s_2 ) - \widetilde { \vecf } \left( s_2 + \frac { \mathcal{L} } 2 \right) \right\|^2 }
	\, d s_1 d s_2
	.
\end{align*}
Consequently we obtain
\begin{align*}
	K_2 \left( \frac { \mathcal{L} } 2 \right)
	= & \
	\int_{ \mathbb{R} / \mathcal{L} \mathbb{Z} }
	\frac { 2 \bar X\left( s_2 + \frac { \mathcal{L} } 2 , s_2 \right)
	\left( \widetilde{ \vectau } \left( s_2 + \frac { \mathcal{L} } 2 \right) - \widetilde{ \vectau } ( s_2 ) \right)
	\cdot
	\left( \widetilde { \vecf } \left( s_2 + \frac { \mathcal{L} } 2 \right) - \widetilde { \vecf } ( s_2 ) \right) }
	{ \left\| \widetilde { \vecf } \left( s_2 + \frac { \mathcal{L} } 2 \right) - \widetilde { \vecf } ( s_2 ) \right\|^2 }
	\, d s_1 d s_2
	\\
	= & \
	\int_{ \mathbb{R} / \mathcal{L} \mathbb{Z} }
	\bar X \left( s_2 + \frac { \mathcal{L} } 2 , s_2 \right)
	\frac d { d s_2 } \log \left\| \widetilde { \vecf } \left( s_2 + \frac { \mathcal{L} } 2 \right) - \widetilde { \vecf } ( s_2 ) \right\|^2
	\\
	= & \
	-
	\int_{ \mathbb{R} / \mathcal{L} \mathbb{Z} }
	\bar X \left( s_2 + \frac { \mathcal{L} } 2 , s_2 \right)
	\frac d { d s_2 } \widetilde X \left( s_2 + \frac { \mathcal{L} } 2 , s_2 \right)
	d s_2
	.
\end{align*}
Combining these,
we have
\[
	K_1 \left( \frac { \mathcal{L} } 2 \right)
	+
	K_2 \left( \frac { \mathcal{L} } 2 \right)
	=
	\frac 12
	\int_{ \mathbb{R} / \mathcal{L} \mathbb{Z} }
	\frac d { d s_2 } \bar X \left( s_2 + \frac { \mathcal{L} } 2 , s_2 \right)^2
	d s_2
	= 0
	.
\]
\par
It follows from
\pref{M} and \pref{D log Ds^2} that
\begin{align*}
	&
	\lim_{ \substack{ \epsilon \to + 0 \\ \delta \to + 0 } }
	\iint_{ \epsilon \leqq | s_1 - s_2 | \leqq \frac { \mathcal{L} } 2 - \delta }
	( X - \widetilde X )
	\frac { \partial^2 } { \partial s_1 \partial s_2 } \left( X + \widetilde X - 2 \log |\Delta s |^2 \right)
	d s_1 d s_2
	\\
	& \quad
	=
	- 2
	\iint_{ ( \mathbb{R} / \mathcal{L} \mathbb{Z} )^2 }
	( X - \widetilde X )
	\left(
	\mathscr{M} ( \vecf ) - \mathscr{M}_0 ( \vecf )
	+ \mathscr{M} ( \widetilde { \vecf } ) - \mathscr{M}_0 ( \widetilde { \vecf } )
	+ \frac 2 { | \Delta s |^2 }
	\right)
	d s_1 d s_2
	.
\end{align*}
\qed
\end{proof}
\begin{thm}
Let $ \vecf $ and $ \widetilde { \vecf } $ represent two embedded close curve with the same total length satisfying
$ \mathcal{E}_0 ( \vecf ) < \infty $,
$ \mathcal{E}_0 ( \widetilde { \vecf } ) < \infty $.
For $ a \in \mathbb{R} / \mathcal{L} \mathbb{Z} $,
set $ \widetilde { \vecf }_a ( \cdot ) = \widetilde { \vecf } ( \cdot + a ) $,
$ \displaystyle{ \widetilde X_a = \log \frac { | \Delta s |^2 } { \| \Delta \widetilde { \vecf }_a \|^2 } } $.
Then,
it holds that
\begin{align*}
	&
	\left| \mathcal{E}_1 ( \vecf ) - \mathcal{E}_1 ( \widetilde { \vecf } ) \right|
	\\
	& \quad
	\leqq
	\left| \mathcal{E}_0 ( \vecf ) - \mathcal{E}_0 ( \widetilde { \vecf } ) \right|
	\\
	& \quad \qquad
	+ \,
	2
	\inf_{ a \in \mathbb{R} / \mathcal{L} \mathbb{Z} }
	\left\{
	\| X - \widetilde X_a \|_{ L^\infty }
	\left( \mathcal{E}_0 ( \vecf ) + \mathcal{E}_0 ( \widetilde { \vecf } ) +4 \right)
	+
	\left|
	\iint_{ ( \mathbb{R} / \mathcal{L} \mathbb{Z} )^2 }
	\frac { X - \widetilde X_a } { | \Delta s |^2 }
	\,
	d s_1 d s_2
	\right|
	\right\}
	,
	\\
	&
	\left| \mathcal{E}_2 ( \vecf ) - \mathcal{E}_2 ( \vecf_0 ) \right|
	\\
	& \quad
	\leqq
	2
	\inf_{ a \in \mathbb{R} / \mathcal{L} \mathbb{Z} }
	\left\{
	\| X - \widetilde X_a \|_{ L^\infty }
	\left( \mathcal{E}_0 ( \vecf ) + \mathcal{E}_0 ( \widetilde { \vecf } ) +4 \right)
	+
	\left|
	\iint_{ ( \mathbb{R} / \mathcal{L} \mathbb{Z} )^2 }
	\frac { X - \widetilde X_a } { | \Delta s |^2 }
	\,
	d s_1 d s_2
	\right|
	\right\}
	.
\end{align*}
\label{modulus_of_continuity}
\end{thm}
\begin{proof}
From Proposition \ref{M_f-M_tilde f},
we know
\begin{align*}
	&
	\left| \mathcal{E}_1 ( \vecf ) - \mathcal{E}_1 ( \widetilde { \vecf } ) \right|
	\\
	& \quad
	\leqq
	\left| \mathcal{E}_0 ( \vecf ) - \mathcal{E}_0 ( \widetilde { \vecf } ) \right|
	\\
	& \quad \qquad
	+ \,
	2 \| X - \widetilde X \|_{ L^\infty }
	\left( \mathcal{E}_0 ( \vecf ) + \mathcal{E}_0 ( \widetilde { \vecf } ) +4 \right)
	+
	2
	\left|
	\iint_{ ( \mathbb{R} / \mathcal{L} \mathbb{Z} )^2 }
	\frac { X - \widetilde X } { | \Delta s |^2 }
	\,
	d s_1 d s_2
	\right|
	.
\end{align*}
This estimate holds if $ \widetilde { \vecf } $ is replaced by $ \widetilde { \vecf }_a $.
Clearly $ \mathcal{E}_i ( \widetilde { \vecf } ) = \mathcal{E}_i ( \widetilde { \vecf }_a ) $ holds.
Hence,
taking the infimum with respect to $ a $,
we obtain the assertion on $ \mathcal{E}_1 $.
The assertion on $ \mathcal{E}_2 $ can be proved in a similar way.
\qed
\end{proof}
\begin{cor}
Let $ \vecf_{\mathrm{circ}} $ represent a round circle with the same total length as that of $ \vecf $,
and let us set $ \displaystyle{ X_{\mathrm{circ}} = \log \frac { | \Delta s |^2 } { \| \Delta \vecf_{\mathrm{circ}} \|^2 } } $.
If $ \mathcal{E}_0 ( \vecf ) < \infty $,
then
\begin{align*}
	&
	\left| \mathcal{E}_1 ( \vecf ) - \mathcal{E}_1 ( \vecf_{\mathrm{circ}} ) \right|
	\\
	& \quad
	\leqq
	\mathcal{E}_0 ( \vecf )
	+
	2 \| X - X_{\mathrm{circ}} \|_{ L^\infty }
	\left( \mathcal{E}_0 ( \vecf ) + 4 \right)
	+
	2
	\left|
	\iint_{ ( \mathbb{R} / \mathcal{L} \mathbb{Z} )^2 }
	\frac { X - X_{\mathrm{circ}} } { | \Delta s |^2 }
	\,
	d s_1 d s_2
	\right|
	,
	\\
	&
	\left| \mathcal{E}_2 ( \vecf ) - \mathcal{E}_2 ( \vecf_{\mathrm{circ}} ) \right|
	\\
	& \quad
	\leqq
	2 \| X - X_{\mathrm{circ}} \|_{ L^\infty }
	\left( \mathcal{E}_0 ( \vecf ) + 4 \right)
	+
	2
	\left|
	\iint_{ ( \mathbb{R} / \mathcal{L} \mathbb{Z} )^2 }
	\frac { X - X_{\mathrm{circ}} } { | \Delta s |^2 }
	\,
	d s_1 d s_2
	\right|
	.
\end{align*}
\end{cor}
\begin{proof}
We apply Theorem \ref{modulus_of_continuity} with $ \widetilde { \vecf } = \vecf_{\mathrm{circ}} $.
Set
\[
	\vecf_{ \mathrm{circ} , a } ( \cdot ) = \vecf_{\mathrm{circ}} ( \cdot + a ) ,
	\quad
	X_{ \mathrm{circ} , a } = \log \frac { | \Delta s |^2 } { \| \Delta \vecf_{ \mathrm{circ} , a } \|^2 }
	.
\]
Then,
$ X_{ \mathrm{circ} , a } $ is independent of $ a $.
And $ \mathcal{E}_0 ( \vecf_{\mathrm{circ}} ) = 0 $.
Hence we obtain the conclusion.
\qed
\end{proof}
\begin{rem}
If $ \vecf \in C^{ 1,1 } ( \mathbb{R} / \mathcal{L} \mathbb{Z} ) $,
then
\[
	\mathcal{E}_1 ( \vecf ) \geqq \mathcal{E}_1 ( \vecf_{\mathrm{circ}} ),
\]
see \cite{IshizekiNagasawaIII}B
\end{rem}
\begin{rem}
Even if $ \mathcal{E}_0 ( \vecf ) = \mathcal{E}_0 ( \widetilde { \vecf } ) $,
it does not necessarily hold that $ X - \widetilde X \equiv 0 $.
For example,
if $ \vecf $ is an image of $ \widetilde { \vecf }$ under some M\"{o}bius transformation,
then the energy for the curves is the same,
but $ X - \widetilde X $ is not necessarily $ 0 $.
Consequently,
it is impossible to estimate $ X - \widetilde X $ by $ \left| \mathcal{E}_0 ( \vecf ) - \mathcal{E}_0 ( \widetilde { \vecf } ) \right| $.
\end{rem}
\section{Difference estimates of energy by M\"{o}bius invariance}
\par
If there exists a M\"{o}bius transformation $ T $ such that $ \vecf = T \widetilde { \vecf } $,
then the left-hand side of the estimates in Theorem \ref{modulus_of_continuity} vanishes,
but the right-hand side does not necessarily vanish. 
In this section,
we estimate the energy difference by use of certain quantities which vanish when two curves are transformed under some M\"{o}bius transformation.
To do this,
we would like to write the difference of energy density by an integration of M\"{o}bius invariance.
Since the energy is scaling invariant,
we may assume that the total lengths of $ \mathrm{Im} \vecf $ and $ \mathrm{Im}  \widetilde { \vecf } $ are the same.
As we saw before,
the difference of $ \mathscr{M}_2 $ is
\begin{align*}
	&
	2
	\left(
	\mathscr{M}_2 ( \vecf )
	-
	\mathscr{M}_2 ( \widetilde { \vecf } )
	\right)
	\\
	& \quad
	=
	2
	\frac { \partial^2 } { \partial s_1 \partial s_2 } \log \frac { \| \Delta \widetilde { \vecf } \|^2 } { \| \Delta \vecf \|^2 }
	\\
	& \quad \qquad
	- \,
	\frac 12
	\left\{
	\frac \partial { \partial s_1 }
	\left(
	\log \| \Delta \vecf \|^2
	-
	\log \| \Delta \widetilde{ \vecf } \|^2
	\right)
	\right\}
	\left\{
	\frac \partial { \partial s_2 }
	\left(
	\log \| \Delta \vecf \|^2
	+
	\log \| \Delta \widetilde{ \vecf } \|^2
	\right)
	\right\}
	\\
	& \quad \qquad
	- \,
	\frac 12
	\left\{
	\frac \partial { \partial s_1 }
	\left(
	\log \| \Delta \vecf \|^2
	+
	\log \| \Delta \widetilde{ \vecf } \|^2
	\right)
	\right\}
	\left\{
	\frac \partial { \partial s_2 }
	\left(
	\log \| \Delta \vecf \|^2
	-
	\log \| \Delta \widetilde{ \vecf } \|^2
	\right)
	\right\}
\end{align*}
for $ \displaystyle{ 0 < | s_1 - s_2 | < \frac { \mathcal{L} } 2 } $.
This formula is in the form from which the singularity at $ \displaystyle{ | s_1 - s_2 | = \frac { \mathcal{L} } 2 } $ is removed,
and is absolutely integrable.
We write curves by a general parameter $ \theta \in \mathbb{R} / \mathbb{Z} $,
not the arch-length parameter.
The reason is as follows.
Let $ T $ be a M\"{o}bius transformation.
In general it does not hold that $ T ( \vecf (s) ) = ( T \vecf ) (s) $,
and therefore the arch-length parameter is not suitable when we compare 
$ \vecf $ with $ T \vecf $.
Strictly speaking,
as functions of $ \theta $,
we must use letters other than $ \vecf $,
$ \widetilde { \vecf } $.
However,
for the sake of simplicity,
we use the same ones.
Since
\begin{gather*}
	\iint_{ ( \mathbb{R} / \mathcal{L} \mathbb{Z} )^2 }
	\frac { \partial^2 } { \partial s_1 \partial s_2 } ( \cdots )
	\, d s_1 d s_2
	=
	\iint_{ ( \mathbb{R} / \mathbb{Z} )^2 }
	\frac { \partial^2 } { \partial \theta_1 \partial \theta_2 } ( \cdots )
	\, d \theta_1 d \theta_2
	,
	\\
	\iint_{ ( \mathbb{R} / \mathcal{L} \mathbb{Z} )^2 }
	\left\{ \frac \partial { \partial s_1 } ( \cdots ) \right\}
	\left\{ \frac \partial { \partial s_2 } ( \cdots ) \right\}
	\, d s_1 d s_2
	=
	\iint_{ ( \mathbb{R} / \mathbb{Z} )^2 }
	\left\{ \frac \partial { \partial \theta_1 } ( \cdots ) \right\}
	\left\{ \frac \partial { \partial \theta_2 } ( \cdots ) \right\}
	\, d \theta_1 d \theta_2
	,
\end{gather*}
we have
\begin{align*}
	&
	2 \left(
	\mathcal{E}_2 ( \vecf )
	-
	\mathcal{E}_2 ( \widetilde { \vecf } )
	\right)
	\\
	& \quad
	=
	\iint_{ ( \mathbb{R} / \mathcal{L} \mathbb{Z} )^2 }
	\left[
	2
	\frac { \partial^2 } { \partial s_1 \partial s_2 } \log \frac { \| \Delta \widetilde { \vecf } \|^2 } { \| \Delta \vecf \|^2 }
	\right.
	\\
	& \quad \qquad
	\left.
	- \,
	\frac 12
	\left\{
	\frac \partial { \partial s_1 }
	\left(
	\log \| \Delta \vecf \|^2
	-
	\log \| \Delta \widetilde{ \vecf } \|^2
	\right)
	\right\}
	\left\{
	\frac \partial { \partial s_2 }
	\left(
	\log \| \Delta \vecf \|^2
	+
	\log \| \Delta \widetilde{ \vecf } \|^2
	\right)
	\right\}
	\right.
	\\
	& \quad \qquad
	\left.
	- \,
	\frac 12
	\left\{
	\frac \partial { \partial s_1 }
	\left(
	\log \| \Delta \vecf \|^2
	+
	\log \| \Delta \widetilde{ \vecf } \|^2
	\right)
	\right\}
	\left\{
	\frac \partial { \partial s_2 }
	\left(
	\log \| \Delta \vecf \|^2
	-
	\log \| \Delta \widetilde{ \vecf } \|^2
	\right)
	\right\}
	\vphantom{ \frac { \| \Delta \widetilde{ \vecf } \|^2 } { \| \Delta \vecf \|^2 } }
	\right]
	d s_1 d s_2
	\\
	& \quad
	=
	\iint_{ ( \mathbb{R} / \mathbb{Z} )^2 }
	\left[
	2
	\frac { \partial^2 } { \partial \theta_1 \partial \theta_2 } \log \frac { \| \Delta \widetilde { \vecf } \|^2 } { \| \Delta \vecf \|^2 }
	\right.
	\\
	& \quad \qquad
	\left.
	- \,
	\frac 12
	\left\{
	\frac \partial { \partial \theta_1 }
	\left(
	\log \| \Delta \vecf \|^2
	-
	\log \| \Delta \widetilde{ \vecf } \|^2
	\right)
	\right\}
	\left\{
	\frac \partial { \partial \theta_2 }
	\left(
	\log \| \Delta \vecf \|^2
	+
	\log \| \Delta \widetilde{ \vecf } \|^2
	\right)
	\right\}
	\right.
	\\
	& \quad \qquad
	\left.
	- \,
	\frac 12
	\left\{
	\frac \partial { \partial \theta_1 }
	\left(
	\log \| \Delta \vecf \|^2
	+
	\log \| \Delta \widetilde{ \vecf } \|^2
	\right)
	\right\}
	\left\{
	\frac \partial { \partial \theta_2 }
	\left(
	\log \| \Delta \vecf \|^2
	-
	\log \| \Delta \widetilde{ \vecf } \|^2
	\right)
	\right\}
	\vphantom{ \frac { \| \Delta \widetilde{ \vecf } \|^2 } { \| \Delta \vecf \|^2 } }
	\right]
	d \theta_1 d \theta_2
	\\
	& \quad
	=
	\lim_{ \epsilon \to + 0 }
	\iint_{ | \theta_1 - \theta_2 | \geqq \epsilon }
	\left[
	2
	\frac { \partial^2 } { \partial \theta_1 \partial \theta_2 } \log \frac { \| \Delta \widetilde { \vecf } \|^2 } { \| \Delta \vecf \|^2 }
	\right.
	\\
	& \quad \qquad
	\left.
	- \,
	\frac 12
	\left\{
	\frac \partial { \partial \theta_1 }
	\left(
	\log \| \Delta \vecf \|^2
	-
	\log \| \Delta \widetilde{ \vecf } \|^2
	\right)
	\right\}
	\left\{
	\frac \partial { \partial \theta_2 }
	\left(
	\log \| \Delta \vecf \|^2
	+
	\log \| \Delta \widetilde{ \vecf } \|^2
	\right)
	\right\}
	\right.
	\\
	& \quad \qquad
	\left.
	- \,
	\frac 12
	\left\{
	\frac \partial { \partial \theta_1 }
	\left(
	\log \| \Delta \vecf \|^2
	+
	\log \| \Delta \widetilde{ \vecf } \|^2
	\right)
	\right\}
	\left\{
	\frac \partial { \partial \theta_2 }
	\left(
	\log \| \Delta \vecf \|^2
	-
	\log \| \Delta \widetilde{ \vecf } \|^2
	\right)
	\right\}
	\vphantom{ \frac { \| \Delta \widetilde{ \vecf } \|^2 } { \| \Delta \vecf \|^2 } }
	\right]
	d \theta_1 d \theta_2
	\\
	& \quad
	=
	\lim_{ \epsilon \to + 0 }
	\iint_{ | \theta_1 - \theta_2 | \geqq \epsilon }
	\left[
	2
	\frac { \partial^2 } { \partial \theta_1 \partial \theta_2 } \log \frac { \| \Delta \widetilde { \vecf } \|^2 } { \| \Delta \vecf \|^2 }
	\right.
	\\
	& \quad \qquad
	\left.
	- \,
	\left\{
	\frac \partial { \partial \theta_1 }
	\left(
	\log \| \Delta \vecf \|^2
	-
	\log \| \Delta \widetilde{ \vecf } \|^2
	\right)
	\right\}
	\left\{
	\frac \partial { \partial \theta_2 }
	\left(
	\log \| \Delta \vecf \|^2
	+
	\log \| \Delta \widetilde{ \vecf } \|^2
	\right)
	\right\}
	\vphantom{ \frac { \| \Delta \widetilde{ \vecf } \|^2 } { \| \Delta \vecf \|^2 } }
	\right]
	d \theta_1 d \theta_2
	.
\end{align*}
\par
In a similar way as \pref{lim I_1},
we have
\[
	\lim_{ \epsilon \to + 0 }
	\iint_{ | \theta_1 - \theta_2 | \geqq \epsilon }
	2
	\frac { \partial^2 } { \partial \theta_1 \partial \theta_2 } \log \frac { \| \Delta \widetilde { \vecf } \|^2 } { \| \Delta \vecf \|^2 }
	\, d \theta_1 d \theta_2
	=
	-8 + 8 = 0
	.
\]
From this,
we indirectly find that
\[
	\lim_{ \epsilon \to + 0 }
	\iint_{ | \theta_1 - \theta_2 | \geqq \epsilon }
	\left\{
	\frac \partial { \partial \theta_1 }
	\left(
	\log \| \Delta \vecf \|^2
	-
	\log \| \Delta \widetilde{ \vecf } \|^2
	\right)
	\right\}
	\left\{
	\frac \partial { \partial \theta_2 }
	\left(
	\log \| \Delta \vecf \|^2
	+
	\log \| \Delta \widetilde{ \vecf } \|^2
	\right)
	\right\}
	d \theta_1 d \theta_2
\]
converges.
Now we assume that $ \| \dot { \vecf } \| \in C^{0,1} $,
$ \| \dot { \widetilde { \vecf } } \| \in C^{0,1} $.
Then,
\begin{align*}
	&
	\iint_{ | \theta_1 - \theta_2 | \geqq \epsilon }
	\left(
	\frac \partial { \partial \theta_1 } \log \| \Delta \vecf \|^2
	\right)
	\left(
	\frac \partial { \partial \theta_2 } \log \| \dot { \vecf } ( \theta_2 ) \|
	\right)
	d \theta_1 d \theta_2
	\\
	& \quad
	=
	\int_{ \mathbb{R} / \mathbb{Z} }
	\left\{
	\int_{ \theta_2 + \epsilon }^{ \theta_2 + 1 - \epsilon }
	\left(
	\frac \partial { \partial \theta_1 } \log \| \Delta \vecf \|^2
	\right)
	d \theta_1
	\right\}
	\left(
	\frac \partial { \partial \theta_2 } \log \| \dot { \vecf } ( \theta_2 ) \|
	\right)
	d \theta_2
	\\
	& \quad
	=
	\int_{ \mathbb{R} / \mathbb{Z} }
	\left(
	\log \frac { \| \vecf ( \theta_2 - \epsilon ) - \vecf ( \theta_2 ) \|^2 } { \| \vecf ( \theta_2 + \epsilon ) - \vecf ( \theta_2 ) \|^2 }
	\right)
	\left(
	\frac \partial { \partial \theta_2 } \log \| \dot { \vecf } ( \theta_2 ) \|
	\right)
	d \theta_2
	.
\end{align*}
It follows from the bi-Lipschitz property of $ \vecf $ that
\[
	\log \frac { \| \vecf ( \theta_2 - \epsilon ) - \vecf ( \theta_2 ) \|^2 } { \| \vecf ( \theta_2 + \epsilon ) - \vecf ( \theta_2 ) \|^2 }
	=
	\log \frac { \| \vecf ( \theta_2 - \epsilon ) - \vecf ( \theta_2 ) \|^2 / \epsilon^2 } { \| \vecf ( \theta_2 + \epsilon ) - \vecf ( \theta_2 ) \|^2 / \epsilon^2 }
\]
is uniformly bounded with respect to $ \epsilon $ and $ \theta $ for small  $ \epsilon > 0 $.
Moreover,
we have
\[
	\log \frac { \| \vecf ( \theta_2 - \epsilon ) - \vecf ( \theta_2 ) \|^2 / \epsilon^2 } { \| \vecf ( \theta_2 + \epsilon ) - \vecf ( \theta_2 ) \|^2 / \epsilon^2 }
	\to
	\log \frac { \| \dot { \vecf } ( \theta_2 ) \|^2 } { \| \dot { \vecf } ( \theta_2 ) \|^2 }
	= 0
\]
as $ \epsilon \to + 0 $ for a.e.\ $ \theta_2 \in \mathbb{R} / \mathbb{Z} $.
Consequently,
applying Lebesgue's convergence theorem,
we obtain
\begin{align*}
	&
	\lim_{ \epsilon \to + 0 }
	\iint_{ | \theta_1 - \theta_2 | \geqq \epsilon }
	\left(
	\frac \partial { \partial \theta_1 } \log \| \Delta \vecf \|^2
	\right)
	\left(
	\frac \partial { \partial \theta_2 } \log \| \dot { \vecf } ( \theta_2 ) \|
	\right)
	d \theta_1 d \theta_2
	\\
	& \quad
	=
	\lim_{ \epsilon \to + 0 }
	\int_{ \mathbb{R} / \mathbb{Z} }
	\left(
	\log \frac { \| \vecf ( \theta_2 - \epsilon ) - \vecf ( \theta_2 ) \|^2 } { \| \vecf ( \theta_2 + \epsilon ) - \vecf ( \theta_2 ) \|^2 }
	\right)
	\left(
	\frac \partial { \partial \theta_2 } \log \| \dot { \vecf } ( \theta_2 ) \|
	\right)
	d \theta_2
	= 0
	.
\end{align*}
Similarly,
we can show
\begin{gather*}
	\lim_{ \epsilon \to + 0 }
	\iint_{ | \theta_1 - \theta_2 | \geqq \epsilon }
	\left(
	\frac \partial { \partial \theta_1 } \log \| \Delta \vecf \|^2
	\right)
	\left(
	\frac \partial { \partial \theta_2 } \log \| \dot { \widetilde { \vecf } } ( \theta_2 ) \|
	\right)
	d \theta_1 d \theta_2
	= 0
	,
	\\
	\lim_{ \epsilon \to + 0 }
	\iint_{ | \theta_1 - \theta_2 | \geqq \epsilon }
	\left(
	\frac \partial { \partial \theta_1 } \log \| \Delta \widetilde { \vecf } \|^2
	\right)
	\left(
	\frac \partial { \partial \theta_2 } \log \| \dot { \vecf } ( \theta_2 ) \|
	\right)
	d \theta_1 d \theta_2
	= 0
	,
	\\
	\lim_{ \epsilon \to + 0 }
	\iint_{ | \theta_1 - \theta_2 | \geqq \epsilon }
	\left(
	\frac \partial { \partial \theta_1 } \log \| \Delta \widetilde { \vecf } \|^2
	\right)
	\left(
	\frac \partial { \partial \theta_2 } \log \| \dot { \widetilde { \vecf } } ( \theta_2 ) \|
	\right)
	d \theta_1 d \theta_2
	= 0
	,
\end{gather*}
and the corresponding result for the limit in which we swap $ \theta_1 $ and $ \theta_2 $.
Clearly it holds that
\[
	\iint_{ | \theta_1 - \theta_2 | \geqq \epsilon }
	\left(
	\frac \partial { \partial \theta_1 } \log \| \dot { \vecf } ( \theta_1 ) \|
	\right)
	\left(
	\frac \partial { \partial \theta_2 } \log \| \dot { \vecf } ( \theta_2 ) \|
	\right)
	d \theta_1 d \theta_2
	= 0
	.
\]
Moreover,
the same result holds for the limit in which one or both of $ \dot { \vecf } ( \theta_i ) $ in the above are replaced by $ \dot { \widetilde { \vecf } } ( \theta_i ) $.
Consequently,
setting
\[
	\mathscr{C} ( \vecf )
	=
	\frac { \| \dot { \vecf } ( \theta_1 ) \| \| \dot { \vecf } ( \theta_2 ) \| }
	{ \| \Delta \vecf \|^2 }
	,
\]
we obtain
\begin{align*}
	&
	-
	\lim_{ \epsilon \to + 0 }
	\iint_{ | \theta_1 - \theta_2 | \geqq \epsilon }
	\left\{
	\frac \partial { \partial \theta_1 }
	\left(
	\log \| \Delta \vecf \|^2
	-
	\log \| \Delta \widetilde{ \vecf } \|^2
	\right)
	\right\}
	\left\{
	\frac \partial { \partial \theta_2 }
	\left(
	\log \| \Delta \vecf \|^2
	+
	\log \| \Delta \widetilde{ \vecf } \|^2
	\right)
	\right\}
	d \theta_1 d \theta_2
	\\
	& \quad
	=
	-
	\lim_{ \epsilon \to + 0 }
	\iint_{ | \theta_1 - \theta_2 | \geqq \epsilon }
	\left\{
	\frac \partial { \partial \theta_1 }
	\left(
	\log \mathscr{C} ( \vecf )
	-
	\log \mathscr{C} ( \widetilde { \vecf } )
	\right)
	\right\}
	\left\{
	\frac \partial { \partial \theta_2 }
	\left(
	\log \mathscr{C} ( \vecf )
	+
	\log \mathscr{C} ( \widetilde { \vecf } )
	\right)
	\right\}
	d \theta_1 d \theta_2
	.
\end{align*}
This argument implies the following theorem.
\begin{thm}
Let $ \vecf $ and $ \widetilde { \vecf } $ satisfy $ \mathcal{E} ( \vecf ) < \infty $,
$ \mathcal{E} ( \widetilde{ \vecf } ) < \infty $,
$ \| \dot { \vecf } \| \in C^{0,1} $,
$ \| \dot { \widetilde { \vecf } } \| \in C^{0,1} $.
Then,
it holds that
\begin{align*}
	&
	\mathcal{E}_1 ( \vecf )
	-
	\mathcal{E}_1 ( \widetilde { \vecf } )
	\\
	& \quad
	=
	\iint_{ ( \mathbb{R} / \mathbb{Z} )^2 }
	\left\{
	\mathscr{C} ( \vecf ) - \mathscr{C} ( \widetilde { \vecf } )
	+
	\frac 12
	\frac { \partial^2 } { \partial \theta_1 \partial \theta_2 }
	\left(
	\log \mathscr{C} ( \vecf )
	-
	\log \mathscr{C} ( \widetilde { \vecf } )
	\right)
	\right\}
	d \theta_1 d \theta_2
	\\
	& \quad
	+ \,
	\frac 12
	\lim_{ \epsilon \to + 0 }
	\iint_{ | \theta_1 - \theta_2 | \geqq \epsilon }
	\left\{
	\frac \partial { \partial \theta_1 }
	\left(
	\log \mathscr{C} ( \vecf )
	-
	\log \mathscr{C} ( \widetilde { \vecf } )
	\right)
	\right\}
	\left\{
	\frac \partial { \partial \theta_2 }
	\left(
	\log \mathscr{C} ( \vecf )
	+
	\log \mathscr{C} ( \widetilde { \vecf } )
	\right)
	\right\}
	d \theta_1 d \theta_2
	\\
	&
	\mathcal{E}_2 ( \vecf )
	-
	\mathcal{E}_2 ( \widetilde { \vecf } )
	\\
	& \quad
	=
	-
	\frac 12
	\lim_{ \epsilon \to + 0 }
	\iint_{ | \theta_1 - \theta_2 | \geqq \epsilon }
	\left\{
	\frac \partial { \partial \theta_1 }
	\left(
	\log \mathscr{C} ( \vecf )
	-
	\log \mathscr{C} ( \widetilde { \vecf } )
	\right)
	\right\}
	\left\{
	\frac \partial { \partial \theta_2 }
	\left(
	\log \mathscr{C} ( \vecf )
	+
	\log \mathscr{C} ( \widetilde { \vecf } )
	\right)
	\right\}
	d \theta_1 d \theta_2
	.
\end{align*}
\label{Diff C}
\end{thm}
\begin{proof}
We have already shown the assertion on $ \mathcal{E}_2 $.
For $ \mathcal{E}_1 $,
using \pref{E0=E1+E2},
we have
\[
	\mathcal{E}_1 ( \vecf )
	-
	\mathcal{E}_1 ( \widetilde { \vecf } )
	=
	\left(
	\mathcal{E}_0 ( \vecf )
	-
	\mathcal{E}_0 ( \widetilde { \vecf } )
	\right)
	-
	\left(
	\mathcal{E}_2 ( \vecf )
	-
	\mathcal{E}_2 ( \widetilde { \vecf } )
	\right)
	.
\]
Now,
we write the first term on the right-hand side by $ \mathscr{C} $.
The cosine of the conformal angle $ \varphi $ of $ \vecf $ is
\[
	\cos \varphi
	=
	\frac { \| \Delta \vecf \|^2 }
	{ 2 \| \dot { \vecf } ( \theta_1 ) \| \| \dot { \vecf } ( \theta_2 ) \| }
	\frac { \partial^2 } { \partial \theta_1 \partial \theta_2 }
	\log \| \Delta \vecf \|^2
	.
\]
Noting
\[
	\frac { \partial^2 } { \partial \theta_1 \partial \theta_2 }
	\log \left( \| \dot { \vecf } ( \theta_1 ) \| \| \dot { \vecf } ( \theta_2 ) \| \right)
	=
	\frac { \partial^2 } { \partial \theta_1 \partial \theta_2 }
	\left(
	\log \| \dot { \vecf } ( \theta_1 ) \|
	+
	\log \| \dot { \vecf } ( \theta_2 ) \|
	\right)
	= 0
	,
\]
we have
\[
	\cos \varphi
	=
	-
	\frac 1 { 2 \mathscr{C} ( \vecf ) }
	\frac { \partial^2 } { \partial \theta_1 \partial \theta_2 }
	\log \mathscr{C} ( \vecf )
	.
\]
Hence,
it holds that
\begin{align*}
	\mathcal{E}_0 ( \vecf )
	= & \
	\iint_{ ( \mathbb{R} / \mathbb{Z} )^2 }
	\frac { 1 - \cos \varphi } { \| \Delta \vecf \|^2 }
	\| \dot { \vecf } ( \theta_1 ) \| \| \dot { \vecf } ( \theta_2 ) \|
	d \theta_1 d \theta_2
	\\
	= & \
	\iint_{ ( \mathbb{R} / \mathbb{Z} )^2 }
	\left( \mathscr{C} ( \vecf )
	+
	\frac 12
	\frac { \partial^2 } { \partial \theta_1 \partial \theta_2 }
	\log \mathscr{C} ( \vecf )
	\right)
	d \theta_1 d \theta_2
	.
\end{align*}
Since we have the corresponding expression for $ \mathcal{E}_0 ( \widetilde { \vecf } ) $,
we obtain the assertion of the Theorem.
\qed
\end{proof}
\par
The M\"{o}bius invariance of the cross ratio implies that of $ \mathscr{C} $.
Consequently,
we can read the M\"{o}bius invariance of $ \mathcal{E}_1 $,
$ \mathcal{E}_2 $ from Theorem \ref{Diff C}.
Moreover,
even if $ \vecf $ cannot be transformed to $ \widetilde { \vecf } $ by any M\"{o}bius transformations,
we can estimate the energy difference by $ \mathscr{C} $ and its derivatives.
\begin{rem}
A sufficient condition for the assumption of Theorem \ref{Diff C} is,
for example,
$ \vecf \in C^{1,1} $,
$ \widetilde {\vecf } \in C^{1,1} $.
\end{rem}
\begin{rem}
At this moment,
the absolute integrability of the integration in principal value in Theorem \ref{Diff C} is not certain.
At a glance,
the integration in the principal value seems asymmetric with respect to $ \theta_1 $ and $ \theta_2 $.
And in fact,
it is symmetric:
\begin{align*}
	&
	\lim_{ \epsilon \to + 0 }
	\iint_{ | \theta_1 - \theta_2 | \geqq \epsilon }
	\left\{
	\frac \partial { \partial \theta_1 }
	\left(
	\log \mathscr{C} ( \vecf )
	-
	\log \mathscr{C} ( \widetilde { \vecf } )
	\right)
	\right\}
	\left\{
	\frac \partial { \partial \theta_2 }
	\left(
	\log \mathscr{C} ( \vecf )
	+
	\log \mathscr{C} ( \widetilde { \vecf } )
	\right)
	\right\}
	d \theta_1 d \theta_2
	\\
	& \quad
	=
	\lim_{ \epsilon \to + 0 }
	\frac 12
	\left[
	\iint_{ | \theta_1 - \theta_2 | \geqq \epsilon }
	\left\{
	\frac \partial { \partial \theta_1 }
	\left(
	\log \mathscr{C} ( \vecf )
	-
	\log \mathscr{C} ( \widetilde { \vecf } )
	\right)
	\right\}
	\left\{
	\frac \partial { \partial \theta_2 }
	\left(
	\log \mathscr{C} ( \vecf )
	+
	\log \mathscr{C} ( \widetilde { \vecf } )
	\right)
	\right\}
	\right.
	\\
	& \quad \qquad
	+ \,
	\left.
	\left\{
	\frac \partial { \partial \theta_1 }
	\left(
	\log \mathscr{C} ( \vecf )
	+
	\log \mathscr{C} ( \widetilde { \vecf } )
	\right)
	\right\}
	\left\{
	\frac \partial { \partial \theta_2 }
	\left(
	\log \mathscr{C} ( \vecf )
	-
	\log \mathscr{C} ( \widetilde { \vecf } )
	\right)
	\right\}
	\right]
	d \theta_1 d \theta_2
	\\
	& \quad
	=
	\lim_{ \epsilon \to + 0 }
	\iint_{ | \theta_1 - \theta_2 | \geqq \epsilon }
	\left\{
	\left(
	\frac \partial { \partial \theta_1 }
	\log \mathscr{C} ( \vecf )
	\right)
	\left(
	\frac \partial { \partial \theta_2 }
	\log \mathscr{C} ( \vecf )
	\right)
	-
	\left(
	\frac \partial { \partial \theta_1 }
	\log \mathscr{C} ( \widetilde { \vecf } )
	\right)
	\left(
	\frac \partial { \partial \theta_2 }
	\log \mathscr{C} ( \widetilde { \vecf } )
	\right)
	\right\}
	d \theta_1 d \theta_2
	.
\end{align*}
\end{rem}
\section{Open problems}
\par
Integration by parts shows that
the integral in the principal value in Theorem \ref{Diff C} is
\begin{align*}
	&
	\lim_{ \epsilon \to + 0 }
	\iint_{ | \theta_1 - \theta_2 | \geqq \epsilon }
	\left\{
	\frac \partial { \partial \theta_1 }
	\left(
	\log \mathscr{C} ( \vecf )
	-
	\log \mathscr{C} ( \widetilde { \vecf } )
	\right)
	\right\}
	\left\{
	\frac \partial { \partial \theta_2 }
	\left(
	\log \mathscr{C} ( \vecf )
	+
	\log \mathscr{C} ( \widetilde { \vecf } )
	\right)
	\right\}
	d \theta_1 d \theta_2
	\\
	& \quad
	=
	\lim_{ \epsilon \to + 0 }
	\left[
	\int_{ \mathbb{R} / \mathbb{Z} }
	\left[
	\left(
	\log \mathscr{C} ( \vecf )
	-
	\log \mathscr{C} ( \widetilde { \vecf } )
	\right)
	\left\{
	\frac \partial { \partial \theta_2 }
	\left(
	\log \mathscr{C} ( \vecf )
	+
	\log \mathscr{C} ( \widetilde { \vecf } )
	\right)
	\right\}
	\right]_{ s_1 = s_2 + \epsilon }^{ s_1 = s_2 + 1 - \epsilon }
	 d s_2
	 \right.
	 \\
	 & \quad \qquad
	 \left.
	 - \,
	\iint_{ | \theta_1 - \theta_2 | \geqq \epsilon }
	\left(
	\log \mathscr{C} ( \vecf )
	-
	\log \mathscr{C} ( \widetilde { \vecf } )
	\right)
	\left\{
	\frac { \partial^2 } { \partial \theta_1 \partial \theta_2 }
	\left(
	\log \mathscr{C} ( \vecf )
	+
	\log \mathscr{C} ( \widetilde { \vecf } )
	\right)
	\right\}
	d \theta_1 d \theta_2
	\right]
	\\
	& \quad
	=
	\lim_{ \epsilon \to + 0 }
	\left[
	\int_{ \mathbb{R} / \mathbb{Z} }
	\left[
	\left(
	\log \mathscr{C} ( \vecf )
	-
	\log \mathscr{C} ( \widetilde { \vecf } )
	\right)
	\left\{
	\frac \partial { \partial \theta_2 }
	\left(
	\log \mathscr{C} ( \vecf )
	+
	\log \mathscr{C} ( \widetilde { \vecf } )
	\right)
	\right\}
	\right]_{ s_1 = s_2 + \epsilon }^{ s_1 = s_2 + 1 - \epsilon }
	 d s_2
	 \right.
	 \\
	 & \quad \qquad
	 \left.
	 + \,
	\iint_{ | \theta_1 - \theta_2 | \geqq \epsilon }
	\left(
	\log \mathscr{C} ( \vecf )
	-
	\log \mathscr{C} ( \widetilde { \vecf } )
	\right)
	\left\{
	\frac { \partial^2 } { \partial \theta_1 \partial \theta_2 }
	\left(
	\log \| \Delta \vecf \|^2
	+
	\log \| \Delta \widetilde { \vecf } \|^2
	\right)
	\right\}
	d \theta_1 d \theta_2
	\right]
	.
\end{align*}
The last expression does not contain derivatives of $ \| \dot { \vecf } \| $ or $ \| \dot { \widetilde { \vecf } } \| $.
If we assume the finiteness of energy on $ \vecf $,
$ \widetilde \vecf $ only,
Theorem \ref{Diff C} seems to be improved so that the integral in the principal value in the theorem is replaced with the above expression.
However,
neither of
\begin{align*}
	&
	\lim_{ \epsilon \to + 0 }
	\int_{ \mathbb{R} / \mathbb{Z} }
	\left[
	\left(
	\log \mathscr{C} ( \vecf )
	-
	\log \mathscr{C} ( \widetilde { \vecf } )
	\right)
	\left\{
	\frac \partial { \partial \theta_2 }
	\left(
	\log \mathscr{C} ( \vecf )
	+
	\log \mathscr{C} ( \widetilde { \vecf } )
	\right)
	\right\}
	\right]_{ s_1 = s_2 + \epsilon }^{ s_1 = s_2 + 1 - \epsilon }
	d \theta_2
	,
	\\
	&
	\lim_{ \epsilon \to + 0 }
	\iint_{ | \theta_1 - \theta_2 | \geqq \epsilon }
	\left(
	\log \mathscr{C} ( \vecf )
	-
	\log \mathscr{C} ( \widetilde { \vecf } )
	\right)
	\left\{
	\frac { \partial^2 } { \partial \theta_1 \partial \theta_2 }
	\left(
	\log \| \Delta \vecf \|^2
	+
	\log \| \Delta \widetilde { \vecf } \|^2
	\right)
	\right\}
	d \theta_1 d \theta_2
\end{align*}
seem to converge.
It would be interesting to determine whether compensating terms exist.
That is,
are there functions $ \mathscr{A} $,
$ \mathscr{B}_\epsilon $ satisfying the following~?
We want these functions to be
\[
	\iint_{ | \theta_1 - \theta_2 | \geqq \epsilon }
	\mathscr{A} ( \theta_1 , \theta_2 ) \, d s_1 d s_2
	+
	\int_{ \mathbb{R} / \mathbb{Z} }
	\mathscr{B}_\epsilon ( \theta_2 ) \, d \theta_2
	= 0
\]
and
\begin{align*}
	&
	\lim_{ \epsilon \to + 0 }
	\int_{ \mathbb{R} / \mathbb{Z} }
	\left(
	\left[
	\left(
	\log \mathscr{C} ( \vecf )
	-
	\log \mathscr{C} ( \widetilde { \vecf } )
	\right)
	\left\{
	\frac \partial { \partial \theta_2 }
	\left(
	\log \mathscr{C} ( \vecf )
	+
	\log \mathscr{C} ( \widetilde { \vecf } )
	\right)
	\right\}
	\right]_{ s_1 = s_2 + \epsilon }^{ s_1 = s_2 + 1 - \epsilon }
	-
	\mathscr{B}_\epsilon
	\right)
	d \theta_2
	,
	\\
	&
	\lim_{ \epsilon \to + 0 }
	\iint_{ | \theta_1 - \theta_2 | \geqq \epsilon }
	\left[
	\left(
	\log \mathscr{C} ( \vecf )
	-
	\log \mathscr{C} ( \widetilde { \vecf } )
	\right)
	\left\{
	\frac { \partial^2 } { \partial \theta_1 \partial \theta_2 }
	\left(
	\log \| \Delta \vecf \|^2
	+
	\log \| \Delta \widetilde { \vecf } \|^2
	\right)
	\right\}
	-
	\mathscr{A}
	\right]
	d \theta_1 d \theta_2
\end{align*}
converge.

\end{document}